\documentclass[a4paper, 11pt]{article}
\usepackage{hyperref}
\usepackage{soul}
\usepackage{appendix}
\usepackage{pifont}
\usepackage{rotating}
\usepackage[doc]{optional}
\usepackage{xcolor}
\usepackage{exscale,relsize}
\usepackage{amsmath}
\usepackage{amsfonts}
\usepackage{amssymb}
\usepackage{calc}
\usepackage{theorem}
\usepackage{pifont}
\usepackage{graphicx}

\definecolor{labelkey}{rgb}{0,0.08,0.45}
\definecolor{rekey}{rgb}{0,0.6,0.0}
\definecolor{Brown}{rgb}{0.45,0.0,0.05}
\newtheorem{theorem}{Theorem}[section]
\newtheorem{lemma}[theorem]{Lemma}

\newtheorem{corollary}[theorem]{Corollary}
\newtheorem{proposition}[theorem]{Proposition}
\newtheorem{definition}[theorem]{Definition}

\theoremstyle{plain}{\theorembodyfont{\rmfamily}
}
\theoremstyle{plain}{\theorembodyfont{\rmfamily}
}
\theoremstyle{plain}{\theorembodyfont{\rmfamily}
}
\theoremstyle{plain}{\theorembodyfont{\rmfamily}
\newtheorem{example}[theorem]{Example}}
\theoremstyle{plain}{\theorembodyfont{\rmfamily}
\newtheorem{remark}[theorem]{Remark}}

\theoremstyle{plain}{\theorembodyfont{\rmfamily}
}

\def\hat{\widehat}

\def\tilde{\widetilde}
\def\Bar{\overline}

\def\R{\mathbb{R}}

\begin{document}

\title{\textsf{Convergent Semidefinite Programming Relaxations for Global Bilevel Polynomial Optimization Problems}}
\author{\textsc{V. Jeyakumar\thanks{%
Department of Applied Mathematics, University of New South Wales, Sydney
2052, Australia. E-mail: v.jeyakumar@unsw.edu.au, Research was partially
supported by the Australian Research Council. }, \quad J.B. Lasserre\thanks{%
LAAS-CNRS and Institute of Mathematics, LAAS, France, E-mail:
lasserre@laas.fr. Research was partially supported by a PGM Grant of the  Foundation Math\'{e}matique Jacques Hadamard.} \quad G. Li\thanks{%
Department of Applied Mathematics, University of New South Wales, Sydney
2052, Australia. E-mail: g.li@unsw.edu.au, Research was partially supported
by  grant from the Australian Research Council. Part of this work was done while the author was visiting LAAS-CNRS in July 2014. } \quad and \quad T. S. Ph\d{a}m\thanks{
Institute of Research and Development, Duy Tan University, K7/25, Quang Trung, Danang, Vietnam,
and Department of Mathematics, University of Dalat, 1 Phu Dong Thien Vuong, Dalat, Vietnam.
E-mail: sonpt@dlu.edu.vn.
This author was partially supported by Vietnam National Foundation for Science and Technology Development (NAFOSTED) grant number 101.04-2013.07.
Part of the work was done while this author was visiting the Department of Applied Mathematics at University of New South Wales, Australia. } }
}
\date{{\bf Second Revised Version: January 9, 2016}}

\maketitle
\vspace{-1.1cm}
\begin{abstract}
In this paper, we consider a bilevel polynomial optimization problem where
the objective and the constraint functions of both the upper and the lower level
problems are polynomials. We present methods for finding its global
minimizers and global minimum using a sequence of semidefinite programming (SDP)
relaxations and provide convergence results for the methods. Our scheme for problems with a convex lower-level problem
involves solving a transformed equivalent single-level problem by a sequence of SDP relaxations; whereas our approach for general problems involving a
non-convex polynomial lower-level problem solves a sequence of approximation
problems via another sequence of SDP relaxations.
\medskip

\end{abstract}

\vspace{-0.6cm}

\vspace{-0.3cm}

\section{Introduction}

Consider the bilevel polynomial optimization problem
\begin{eqnarray*}
(P) &\displaystyle\min_{x\in \mathbb{R}^{n},y\in \mathbb{R}^{m}}&f(x,y) \\
&\mbox{subject to }&g_i(x,y)\leq 0,\ i=1,\ldots ,s, \\
&&y\in Y(x):=\displaystyle\mathrm{argmin}_{w\in \mathbb{R}%
^{m}}\{G(x,w):h_{j}(w)\leq 0,j=1,\ldots ,r\},
\end{eqnarray*}%
where $f:\mathbb{R}^{n}\times \mathbb{R}^{m}\rightarrow \mathbb{R}$, $g_{i}:%
\mathbb{R}^{n} \times \mathbb{R}^m \rightarrow \mathbb{R}$, $G:\mathbb{R}%
^{n}\times \mathbb{R}^{m}\rightarrow \mathbb{R}$ and $h_{j}: \mathbb{R}^{m}\rightarrow \mathbb{R}$ are all polynomials with
real coefficients, and we make the blanket assumption that the feasible set of (P) is nonempty, that is,
$\{(x,y) \in \mathbb{R}^n \times \mathbb{R}^m: g_i(x,y) \le 0, i=1,\ldots,s, \ y \in Y(x)\} \neq \emptyset$.

Bilevel optimization provides mathematical models for hierarchical
decision making processes where the follower's decision depends on the
leader's decision. More precisely, if $x$ and $y$ are the decision
variables of the leader and the follower respectively then the problem (P)
represents the so-called optimistic approach to the leader and follower's
game in which the follower is assumed to be co-operative and so, the leader
can choose the solution with the lowest cost. We note that, there is another approach, called
pessimistic approach, which assumes that the follower may not be co-operative
and hence the leader will need to prepare for the worst cost (see for example \cite{Dempe,bilevel}).  %where the bilevel problem is
% replaced by
% \begin{eqnarray*}
% &\displaystyle\min_{x\in \mathbb{R}^{n}}& f(x,y) \\
% &\mbox{subject to }& g_i(x,y)\leq 0,\ i=1,\ldots ,s, \\
% &&y\in Y(x):=\displaystyle\mathrm{argmin}_{w\in \mathbb{R}%
% ^{m}}\{G(x,w):h_{j}(w)\leq 0,j=1,\ldots ,r\},
% \end{eqnarray*}
% formed by replacing the constraint $y\in Y(x)$ in (P) with $y\in Y^{\prime
% }(x):=\displaystyle\mathrm{argmax}_{w\in \mathbb{R}^{m}}\{G(x,w):h_{j}(w)%
% \leq 0,j=1,\ldots ,r\}$.

The bilevel optimization problem (P) also requires that the constraints of the lower level problem are independent of the upper level decision variable $x$ (i.e. the functions $h_{j}$ do not depend on $x$). This independence assumption guarantees that the optimal value function of the lower level problem is automatically continuous, and so, plays an important role later in establishing convergence of our proposed approximation schemes for finding global optimal solutions of (P). A discussion on this assumption and its possible relaxation is given in Remark \ref{remark:3} in Section 4 of the paper.

As noted in \cite{Ye}, the models of the form $(P)$ cover
the situations in which the leader can only observe the outcome of the
follower's action but not the action itself, and so, has important applications in economics such as the so-called moral
hazard model of the principal-agent problem. In particular, in the special
case where $g_i$ depends only on $x$, the sets $\{x\in \mathbb{R}^{n}:g_{i}(x)\leq
0\}$ and $\{w\in \mathbb{R}^{m}:h_{j}(w)\leq 0\}$ are both convex sets,
problem (P) has been studied in \cite{Ye} and a smoothing projected gradient
algorithm has been proposed to find a stationary point of problem (P). On the other hand, the functions $%
f,g_{i},G,h_{j}$ of  (P) in \cite{Ye} are allowed to be continuously differentiable functions which may not
be polynomials in general. For applications and recent developments of solving more
general bilevel optimization problems, see \cite%
{Bilevel_book,Marcotte,Dempe0,Dempe,Vicente}.

In this paper, in the interest of simplicity, we focus on the {\it optimistic approach} to the hierarchical decision making process
and develop methods for finding a global minimizer and global minimum of $(P)$. We make the following key contributions to bilevel optimization.
\vspace{-0.4cm}
\begin{itemize}
\item{\bf A novel SDP hierarchy for bilevel polynomial problems.} We propose general purpose schemes for finding global solutions of the bilevel polynomial optimization problem (P) by solving hierarchies of
semidefinite programs and establish convergence of the schemes. Our approach makes use of the known techniques of bilevel optimization
and the recent developments of (single-level) polynomial optimization, such as the sums-of-squares decomposition and semidefinite programming hierarchy, and does not use any discretization or branch-and-bound techniques as in \cite{Floudas2001,Barton,bilevel}.

\item{\bf Convex lower-level problems: Convergence to global solutions.}  We first transform the bilevel polynomial optimization problem (P) with a {\em convex lower-level problem} into an equivalent single
level nonconvex polynomial optimization problem. We show that the values of the standard semidefinite programming relaxations of the transformed single level problem converge to the global optimal value of the bilevel problem (P) under a technical assumption that is commonly used in polynomial optimization (see \cite{Lasserre_book} and other references therein).

\item{\bf Non-convex lower-level problems: A new convergent approximation scheme.} By examining a sequence of $\epsilon $-approximation (single-level) problems of the bilevel problem (P) with a {\em not necessarily convex lower level problem}, we present another convergent  sequence of SDP relaxations of (P) under suitable conditions. Our approach extends the sequential SDP relaxations, introduced in \cite%
{Lasserre2010} for parameterized single-level polynomial problems, to bilevel polynomial optimization problems.
\end{itemize}

\vspace{-0.3cm} It is important to note that local bilevel optimization techniques, studied extensively in the literature \cite{Bilevel_book,Dempe0}, apply to broad classes of nonlinear bilevel optimization problems. In the present work, we employ some basic tools and techniques of semi-algebraic geometry to achieve convergence of our semidefinite programming hierarchies of global nonlinear bilevel optimization problems, and so our approaches are limited to studying the class of polynomial bilevel optimization problems.

Moreover, due to the limitation of the SDP programming solvers, our proposed scheme  can be used to solve problems with small or moderate size and it may not be able to compete with the ad-hoc (but computationally tractable) techniques, such as branch-and-bound methods and discretization schemes. For instance, underestimation and branch-and-bound techniques were used in \cite{mix,Floudas2001,Barton} and
a generalized semi-infinite programming reformulation together with a discretization technique was employed in
\cite{bilevel}. See http://bilevel.org/ for other references of computational methods of bilevel optimization.

However, it has recently been shown that, by exploiting sparsity and symmetry, large size problems can be solved efficiently and various numerical packages have been built to solve real-life problems such as the sensor network localization problem \cite{kim}. We leave the study of solving large size bilevel problems for future research as it is beyond the scope of this paper. \setcounter{equation}{0}

The outline of the paper is as follows. Section 2 gives preliminary results on polynomials and continuity properties of the solution map of the lower-level problem of (P). Section 3 presents convergence of our sequential SDP relaxation scheme for solving the problem (P) with a convex lower-level problem. Section 4 describes another sequential SDP relaxation scheme and its convergence for solving the general problem (P) with a not necessarily convex lower-level problem.
Section 5 reports results of numerical implementations of the proposed methods for solving some bilevel optimization test problems. The appendix provides details of various technical results of semi-algebraic geometry used in the paper and also proofs of certain technical results.

\section{Preliminaries}
% \begin{itemize}
% \item {\bf Assumption~1:} The set $K:=\{x\in\R^n|\;g_i(x,y)\le 0\;\mbox{ as }\;i=1,\ldots\;r\mbox{ and }\;h_j(x)=0\;\mbox{ as}\;j=1,\ldots,s\}$ is compact and the function $f$ is
% continuous on $\R^n\times\R^l$.
% \item {\bf Assumption 2:} There are constants $L, \delta >0$ such that
% \begin{equation}\label{eq:assumption2}
% \|f(x,u)-f(x,\overline{u})\|\le L \|u-\overline{u}\|
% \end{equation}
% for all $x\in K$ and for all $u$ with $\|u-\overline{u}\|\le\delta $.
% \end{itemize}
We begin by fixing notation, definitions and preliminaries. Throughout this
paper $\mathbb{R}^{n}$ denotes the Euclidean space with dimension $n$. The
inner product in $\mathbb{R}^{n}$ is defined by $\langle x,y\rangle :=x^{T}y$
for all $x,y\in \mathbb{R}^{n}$. The open (resp. closed) ball in $\mathbb{R}^n$ centered at $x$ with radius $%
\rho$ is denoted by ${\mathbb{B}}(x,\rho)$ (resp. $\overline{\mathbb{B}}(x,\rho)$).
% We also use $\mathbb{B}$ (resp. $\overline{\mathbb{B}}$) to denote the open (resp.
%closed) unit ball centered at the origin.
The non-negative
orthant of $\mathbb{R}^{n}$ is denoted by $\mathbb{R}_{+}^{n}$ and is
defined by $\mathbb{R}_{+}^{n}:=\{(x_{1},\ldots ,x_{n})\in \mathbb{R}^{n}\
|\ x_{i}\geq 0\}$. Denote by $\mathbb{R}[x]$ the ring of polynomials in $%
x:=(x_{1},x_{2},\ldots ,x_{n})$ with real coefficients. For a polynomial $f$
with real coefficients, we use $\mathrm{deg}\,f$ to denote the degree of $f$%
. For a differentiable function $f$ on $\mathbb{R}^n$, $\nabla f$ denotes its derivative. For a  differentiable function $g:\mathbb{R}^n \times \mathbb{R}^m \rightarrow \mathbb{R}$, we use $\nabla_x g$ (resp. $\nabla_y g$) to denote the derivative of $g$ with respect to the first variable (resp. second variable).
%A symmetric $n\times n$ matrix $A$ is said to be \emph{positive definite},
%denoted by $A\succ 0,$ whenever $x^{T}Ax>0$ for all $x\in \mathbb{R}%
%^{n},x\neq 0.$
We also use $\mathbb{N}$ (resp. $\mathbb{N}_{>0})$ to denote all the nonnegative (resp. positive) integers. Moreover, for any integer $t$, let
$\mathbb{N}^n_t:=\{{\bf \alpha} \in \mathbb{N}^n: \sum_{i=1}^n \alpha_i \le t\}$. For a set $A
$ in $\mathbb{R}^n$, we use $\mathrm{cl}(A)$ and $\mathrm{int}(A)$ to denote
the closure and interior of $A$. For a given point $x$, the distance from
the point $x$ to a set $A$ is denoted by $d(x,A)$ and is defined by $%
d(x,A)=\inf\{\|x-a\|: a \in A\}$.

%The gradient and the Hessian of a real polynomial $f\in \mathbb{R}[\underline{x}]$ at a point $x^*$ are denoted by $\nabla f(x^*)$ and $\nabla^2 f(x^*)$ respectively.
%Moreover, for a function $L \colon \mathbb{R}^n \times \mathbb{R}^m \rightarrow \mathbb{R}$, we use $\nabla^2_{xx}L(x^*,\lambda^*)$ to denote the Hessian of $L$ at
%$(x^*, \lambda^*)\in \mathbb{R}^n \times \mathbb{R}^m$ with respect to the variable $x$.

We say that a real polynomial $f\in \mathbb{R}[x]$ is \emph{sum-of-squares}
(SOS) if there exist real polynomials $f_{j},j=1,\ldots ,r,$ such that $%
f=\sum_{j=1}^{r}f_{j}^{2}$. The set of all sum-of-squares real polynomials
in $x$ is denoted by $\Sigma ^{2}[x].$ Moreover, the set of all
sum-of-squares real polynomials in $x$ with degree at most $d$ is denoted by
$\Sigma _{d}^{2}[x].$  We also recall some notions and results of semi-algebraic functions/sets, which can be found in \cite{Bochnak1998, Dries1996}.
\begin{definition}{\bf (Semi-algebraic sets and functions)} A subset of $\mathbb{R}^n$ is called {\em semi-algebraic} if it is a finite union of sets of the form
$\{x \in \mathbb{R}^n \ : \ f_i(x) = 0, i = 1, \ldots, k; f_i(x) > 0, i = k + 1, \ldots, p\},$
where all $f_{i}$ are real polynomials. If $A \subset \Bbb{R}^n$ and $B \subset \Bbb{R}^p$ are  semi-algebraic sets then the map $f \colon A \to B$ is said to be {\em semi-algebraic} if its graph
$\{(x, y) \in A \times B \ : \ y = f(x)\}$
is a semi-algebraic subset in $\Bbb{R}^n\times\Bbb{R}^p.$
\end{definition}

Semi-algebraic sets and functions are important classes of sets and functions and they have important applications in nonsmooth optimization (for a recent development, see \cite{Bolte}). In particular, they enjoy a number of remarkable properties. Some of these  properties, which are used later in the paper, have
been summarized in the Appendix A for the convenience of the reader.

We now present a preliminary result on H\"{o}lder continuity
of the solution mapping of the lower level problem. As a consequence, we
provide an existence result of the solution of a bilevel polynomial
optimization problem (P).
%by studying a function called partial optimality function.
%Define the partial optimal value function $P$ by
%\[
%P(x):=\min_{y \in Y(x)} f(x,y) \mbox{ for all } x \in \mathbb{R}^n.
%\]
%Then, the original bilevel optimization problem can be equivalently rewritten as the following single level optimization problem
%\[
%\min_{x \in \mathbb{R}^n}\{P(x) : g_i(x,y) \le 0, i=1,\ldots,s\}.
%\]
%Thus, the existence of the solution of a bilevel polynomial optimization problem (P) can be achieved by examining the continuity of the partial optimal value function.
%
%Next, we discuss the continuity of this partial optimal value function. To do this, we first examine the H\"{o}lder continuity of the solution mapping of the lower level problem.

Let $F \colon \mathbb{R}^n \rightrightarrows \mathbb{R}^m$ be a set-valued mapping and let $\bar y \in F(\bar x)$.
Recall that $F$ is said to be H\"{o}lder continuous (calm) at $(\bar x,\bar y)$ with exponent $\tau
\in (0,1]$ if there exist $\delta,\epsilon,c>0$  such that
\begin{equation*}
d\left(y, F(\bar x) \right) \le c\,\|x-\overline{x}\|^{\tau} \mbox{ for all }
y \in F(x) \cap \mathbb{B}_{\mathbb{R}^m}(\bar y, \epsilon) \mbox{ and } x
\in \mathbb{B}_{\mathbb{R}^n}(\bar x,\delta).
\end{equation*}
In the case when $\tau=1$, this property is often refereed as calmness and
has been well-studied in nonsmooth analysis (see for example \cite%
{Clarke}). We first see that even in the case, where $G$ is a continuously differentiable function and the set $\{y
\in \mathbb{R}^m: h_j(y) \le 0\}$ is compact, the solution map $%
Y\colon \mathbb{R}^n \rightrightarrows \mathbb{R}^m$ of the lower level
problem $Y(x):=\displaystyle \mathrm{argmin}_{y \in \mathbb{R}^m}\{G(x,y):
h_j(y) \le 0,j=1,\ldots,r\}$ is not necessarily H\"{o}lder continuous for any
exponent $\tau>0$.
\begin{example}{\bf (Failure of H\"{o}lder continuity for solution map of the lower level problem: non-polynomial case)}
Let $f: \mathbb{R} \rightarrow \mathbb{R}$ be defined by
\[
f(y)=\left\{\begin{array}{ccc}
e^{-\frac{1}{y^2}} & \mbox{ if } & y \neq 0\\
0 & \mbox{ if } & y=0.
\end{array} \right.
\]
Consider the solution mapping $Y(x)={\rm argmin}_y\{G(x,y): y^2 \le 1\}$ for all $x \in [-1,1]$, where $G(x,y)=(x-f(y))^2$. Then, it can be verified that $G$ is a
continuously differentiable function (indeed it is a $C^{\infty}$ function) and
$$Y(x)=\left\{\begin{array}{lll}
\{\pm \sqrt{\frac{1}{-\ln x}}\} & \mbox{ if } & x \in (0,1], \\
\{0\}  & \mbox{ if } & x \in [-1,0].
             \end{array}
\right.$$
We now see that the solution mapping $Y$ is not H\"{o}lder continuous at $0$ with exponent $\tau$ for any $\tau \in (0,1]$. To see this, let $x_k=e^{-k} \rightarrow 0$ and $y_k=\sqrt{\frac{1}{k}} \in Y(x_k)$. Then, for any $\tau \in (0,1]$,
\[
\frac{|x_k|^{\tau}}{d(y_k, Y(0))} = \frac{e^{-\tau k}}{\sqrt{\frac{1}{k}}}=\frac{\sqrt{k}}{e^{\tau k}} \rightarrow 0.
\]
So, the solution mapping is not H\"{o}lder continuous at $0$ for any $\tau \in (0,1]$.
\end{example}

The H\"{o}lder continuity of the solution set of general parametric optimization problems
has been established under suitable regularity conditions, for example see \cite{Holder,Robinson}. This property plays an important role in
establishing the existence of solutions for bilevel programming problems and equilibrium problems (see for example \cite{existence} and Corollary \ref{cor:existence}). Next, we show that, the solution map of a lower level problem of a bilevel
polynomial optimization problem is always H\"{o}lder continuous with an explicit
exponent which depends only on the degree of the polynomial involved and the
dimension of the underlying space. This result is based on our recent
established \L ojasiewicz inequality for nonconvex polynomial systems in
\cite{Li_Mor_Pham_2013}.

For $m,d \in \mathbb{N}$, denote
\begin{equation}  \label{kappa}
\quad R(m, d) :=%
\begin{cases}
1 & \mbox{ if }\;d=1, \\
d(3d-3)^{m-1} & \mbox{ if }\;d\ge 2.%
\end{cases}%
\end{equation}

\begin{theorem}
\textbf{(H\"older continuity of solution maps in the lower level problem:
polynomial case)}\label{th:stability} \label{Theorem1} Let $h_j$, $j=1,\ldots,r$ and $G$ be polynomials with real coefficients. Denote $d :=\max\{\mathrm{deg%
}h_j, \mathrm{deg} G(x,\cdot)\}$. Suppose that $F := \{y \in
\mathbb{R}^m: h_j(y) \le 0\}$ is compact.  Then, the solution map $Y\colon \mathbb{R}%
^n \rightrightarrows \mathbb{R}^m$ in the lower level problem $Y(x):=%
\displaystyle \mathrm{argmin}_{y \in \mathbb{R}^m}\{G(x,y): h_j(y) \le
0,j=1,\ldots,r\}$ satisfies the following H\"older continuity property at
each point $\bar{x} \in \mathbb{R}^n$: for any $\delta > 0$, there is a
constant $c > 0$ such that
\begin{equation}  \label{hs}
Y(x)\subset Y(\overline{x})+ c\,\|x-\overline{x}\|^{\tau}\,\overline{\mathbb{%
B}}_{\mathbb{R}^m}(0,1)\;\mbox{ whenever }\;\|x-\overline{x}\|\le\delta,
\end{equation}
for some $\tau \in [\tau_0,1]$ with $\tau_0=\max\{\frac{1}{R(m+r+1,d+1)},\frac{%
2}{R(m+r,2d)}\}$. In particular, $Y$ is H\"{o}lder continuous at $\bar x$ with
exponent $\tau_0$ for any $\bar{x} \in \mathbb{R}^n$.
\end{theorem}

\begin{proof}
For any fixed $x \in\mathbb{R}^n$, define $\Phi(x)=\min_{y \in \mathbb{R}^m}\{G(x,y): h_j(y) \le 0,j=1,\ldots,r\}$  and
let
$$\Phi_x(y):= \sum_{i=1}^r\big[h_j(y)\big]_+ \ +|\Phi(x)-G(x,y)|.$$
Then, for all fixed $x$,
\begin{eqnarray*}
\{y\in\R^n|\;\Phi_x(y)=0\} &=& Y(x) \\
&=& \big\{y\in\mathbb{R}^m \big|\;  h_j(y) \le 0\;\mbox{ as }\;j=1,\ldots,s,\mbox{ and }\;\Phi(x)-G(x,y)=0\big\}.
\end{eqnarray*}
% $$
%and observe that  Let $\bar{u}$ be an arbitrary point in $\mathbb{R}^n.$
Note that $F$ is compact. Now, the \L ojasiewicz inequality for nonconvex polynomial systems \cite[Corollary 3.8]{Li_Mor_Pham_2013} gives that there is a constant $c_0  > 0$ such that
\begin{equation}\label{eq:qq}
d\big(y,Y(\overline{x})\big)\le c_0\,\Phi_{\overline{x}}(y)^{\tau}\;\mbox{ for all } y \in F,
\end{equation}
for some $\tau \in [\tau_0,1]$ with $\tau_0=\max\{\frac{1}{R(m+r+1,d+1)},\frac{2}{R(m+r,2d)}\}$. Further, there is a constant $L > 0$ such that
\begin{equation}\label{eq:assumption2}
|G(x,y)-G(\overline{x},y)|\le L \|x-\overline{x}\|
\end{equation}
for all $y\in F$ and for all $x$ with $\|x-\overline{x}\|\le\delta$.
Denote $c:=(2\beta^{-1} L)^{\tau}$ with $\beta:=c_0^{-\frac{1}{\tau}}> 0$. For any $y\in Y(x)$ we select now $\overline{y} \in Y(\overline{x})$ satisfying $\|y-\overline{y}\|=d(y,Y(\overline{x}))$.
To finish the proof, it suffices to show that
\begin{equation}\label{eq:claim}
\|y-\overline{y}\|\le c\,\|x-\overline{x}\|^{\tau}.
\end{equation}
To see this, note that
$|\Phi(\overline{x})-G(\overline{x},y)|=\Phi_{\overline{x}}(y) \ge\beta d\big(y,Y(\overline{x})\big)^{\frac{1}{\tau}}=\beta \|y-\overline{y}\|^{\frac{1}{\tau}}$.
Since $\overline{y} \in Y(\overline{x})$, it follows that $G(\overline{x},\overline{y})=\Phi(\overline{x})\le G(\overline{x},y)$, and hence
\begin{equation}\label{eq:ok}
\|y-\overline{y}\|^{\frac{1}{\tau}}\le\beta^{-1}|\Phi(\overline{x})-G(\overline{x},y)|=\beta^{-1}\big(G(\overline{x},y)-G(\overline{x},\overline{y})\big).
\end{equation}
Furthermore, as $y\in Y(x)$, $G(x,y)\le G(x,\overline{y})$, and therefore (\ref{eq:assumption2}) gives us that
\begin{eqnarray*}
G(\overline{x},y)-G(\overline{x},\overline{y})&=&\big(G(\overline{x},y)-G(x,y)\big)+\big(G(x,y)-G(x,\overline{y})\big)+\big(G(x,\overline{y})-G(\overline{x},\overline{y})\big)\\
&\le&\big(G(\overline{x},y)-G(x,y)\big)+\big(G(x,\overline{y})-G(\overline{x},\overline{y})\big)\\
&\le& 2 L \|x-\overline{x}\|\;\mbox{ as }\;y,\overline{y} \in F.
\end{eqnarray*}
This together with (\ref{eq:ok}) yields $$\|y-\overline{y}\|^{\frac{1}{\tau}}\le\beta^{-1}\big(G(\overline{x},y)-G(\overline{x},\overline{y}) \big)\le 2\beta^{-1}L \|x-\overline{x}\|.$$
Thus
$$d\big(y,Y(\overline{x})\big)=\|y-\overline{y}\|\le c\,\|x-\overline{x}\|^{\tau},$$
which verifies  (\ref{eq:claim}) and completes the proof of the theorem.
\end{proof}

In general, our lower estimate of the exponent $\tau$ will not be tight. We
present a simple example to illustrate this.
\begin{example}
Consider the solution mapping $Y(x)={\rm argmin}_{y \in \mathbb{R}}\{(x-y^2)^2: y^2 \le 1\}$ for all $x \in [-1,1]$. Clearly,
$$Y(x)=\left\{\begin{array}{lll}
\{\pm \sqrt{x}\} & \mbox{ if } & x \in [0,1], \\
\{0\}  & \mbox{ if } & x \in [-1,0).
             \end{array}
\right.$$
So, the solution mapping is H\"{o}lder continuous at $0$ with exponent $1/2$. On the other hand, our lower estimate gives $\tau_0=1/84$. So, the lower estimate is not tight.
\end{example}

\begin{corollary}{\bf (Existence of global minimizer)}\label{cor:existence}
\label{cor:1} For the bilevel polynomial optimization problem (P), let $%
K=\{(x,y) \in \mathbb{R}^n \times \mathbb{R}^m : g_i(x,y) \le 0\}$ and $F=\{w
\in \mathbb{R}^m: h_j(w) \le 0\}$. Suppose that $K_1=\{x \in \mathbb{R}^n: (x,y) \in K \mbox{ for some } y \in \mathbb{R}^m\}$ and $F$ are compact sets. Then, a global
minimizer for (P) exists.
\end{corollary}

\begin{proof}
Denote the optimal value of problem (P) by ${\rm val}(P)$. Let $(x_k,y_k)$ be a minimizing sequence for the bilevel polynomial optimization problem (P) in the sense that
$g_i(x_k,y_k) \le 0$, $i=1,\ldots,s$, $h_j(y_k) \le 0$, $j=1,\ldots,r$, $y_k \in Y(x_k)$ and $f(x_k,y_k) \rightarrow {\rm val}(P)$.
Clearly, $(x_k,y_k) \in K$ (and so, $x_k \in K_1$) and $y_k \in F$. By passing to a subsequence, we may
assume that $(x_k,y_k) \rightarrow (\bar x,\bar y) \in K_1 \times F$. By continuity, we have $f(\bar x,\bar y)={\rm val}(P)$. To see the conclusion,
it suffices to show that $\bar y \in Y(\bar x).$
Denote $\epsilon_k=\|x_k-\bar x\| \rightarrow 0$.
 Then, by Theorem \ref{Theorem1}, there is $c > 0$ such that
$$Y(x_k) \subseteq Y(\bar x) + c\, \epsilon_k^\tau \, \Bar{\mathbb{B}}_{\mathbb{R}^m}(0,1) \quad \mbox{ for all } \quad k \in \mathbb{N}.$$
As $y_k \in Y(x_k)$, there exists $y_k' \in Y(\bar x)$ such that
\begin{equation}\label{eq:88}
\|y_k- y_k'\| \le 2c  \epsilon_k^\tau \rightarrow 0.
\end{equation}
Note that $Y(\bar x) \subseteq F$, $Y(\bar x)$ is a closed set and $F$ is compact. It follows that $Y(\bar x)$ is also a compact set. Passing to the limit in (\ref{eq:88}), we see that $\bar y \in Y(\bar x)$.
% From the preceding theorem, the partial optimal value function
% \[
% P(x):=\min_{y \in Y(x)} f(x,y) \mbox{ for all } x \in \mathbb{R}^n,
% \]
% is lower semicontinuous. As the original bilevel optimization problem can be equivalently rewritten as the following single level optimization problem
% \[
% \min_{x \in \mathbb{R}^n}\{P(x) : g_i(x,y) \le 0, i=1,\ldots,s\}.
% \]
 So, a global minimizer for problem (P) exists.
\end{proof}

\bigskip

% Let $v \in \mathbb{N}$ and let $f_{1},\ldots ,f_{p}$ be real polynomials in $w$ on $\mathbb{R}%
% ^{v}$, the set $\mathbf{M}(f_{1},\ldots ,f_{p})$ is defined by
% \begin{equation}
% \mathbf{M}(f_{1},\ldots ,f_{p}):=\{\sigma _{0}+\sigma _{1}f_{1}+\cdots
% +\sigma_{p}f_{p}\ |\ \sigma _{i}\in \Sigma ^{2}[w],i=0,1,\ldots ,p\}.
% \label{eq:M}
% \end{equation}%
%
%\end{definition}
%A quadratic module is a subset of polynomials that are non-negative on the set
%$\{x \in {\Bbb R}^n \ | \ g_i(x,y) \le 0, i  =1, \ldots, m\}$ and it possess a very nice certificate for this property.
%
% The following archimedean condition (see \cite{Lasserre2009,Put}) has played a key role in the study of polynomial optimization.
% \begin{definition}{\bf (Archimedean)}\label{A}
% The quadratic module $\mathbf{M}(-g_1, \ldots, -g_m)$ is called {\em Archimedean}  if there exists
% $p \in \mathbf{M}(-g_1, \ldots, -g_m)$ such that  $\{x: p(x) \ge 0\}$ is compact.
% \end{definition}
% When the quadratic module $\mathbf{M}(-g_1, \ldots, -g_m)$ is Archimedean, we have the
% following important characterization of positivity of a polynomial over a semialgebraic set.

The following lemma of Putinar (\cite{Put}), which provides a
characterization for positivity of a polynomial over a system of polynomial
inequalities, can also be regarded as a polynomial analog of Farkas' lemma \cite{dinh-jeya}. This lemma has been extensively used in polynomial optimization \cite{Lasserre_book}
and plays a key role in the convergence analysis of our proposed method later on.
%Recall that the nonhomogeneous
%Farkas' lemma can be stated as follows.

\begin{lemma}
\label{putinar} \textbf{(Putinar's Positivstellensatz)\cite{Put}} Let $f_0$
and $f_{i}$, $i=1,\ldots ,p$ be real polynomials in $w$ on $\mathbb{R}^v$. Suppose
that there exist $R>0$ and sums-of-squares polynomials $\hat{\sigma}_1,\ldots,\hat{\sigma}_p \in \Sigma^2[w]$
such that $R-\|w\|^2=\hat{\sigma} _{0}(w)+\sum_{i=1}^p \hat{\sigma}_{i}f_{i}(w) \mbox{ for all } w \in \mathbb{R}^v.$
If $f_0(w)>0$ over the set $\{w \in \mathbb{R}^v: f_i(w) \geq 0, i=1,\ldots,p\}$, then there exist $\sigma _{i}\in \Sigma^{2}[w]$, $i=0,1,\ldots
,p$ such that $f_0=\sigma _{0}+\sum_{i=1}^p \sigma_{i}f_{i}.$
%\end{itemize}
\end{lemma}

The following assumption plays a key role throughout the paper.

\textbf{Assumption 2.1:} There exist $R_1,R_2>0$ such that the quadratic polynomials $%
(x,y) \mapsto R_1-\Vert (x,y)\Vert ^{2}$ and $%
y \mapsto R_2-\Vert y\Vert ^{2}$ can be written as
\[
R_1-\Vert (x,y)\Vert ^{2}=\sigma_{0}(x,y)-\sum_{i=1}^{s}\sigma_{i}(x,y)g_i(x,y) \ \
\mbox{ and }  \ \
R_2-\Vert y \Vert ^{2}= \bar \sigma_{0}(y)-\sum_{j=1}^{r}\bar \sigma
_{j}(y)h_{j}(y),
\]
for some sums-of-squares polynomials $\sigma_{0},\sigma _{1},\ldots ,\sigma
_{s}\in \Sigma^2 [x,y]$ and sums-of-squares polynomials $\bar \sigma_0,\bar \sigma_1, \ldots, \bar \sigma_r \in \Sigma^2[y]$.

We note that Assumption 2.1 implies that both $K=\{(x,y)\in \mathbb{R}%
^{n} \times \mathbb{R}^m:g_i(x,y)\leq 0,i=1,\ldots ,s\}$ and $F=\{y\in \mathbb{R}%
^{m}:h_{j}(y)\leq 0,j=1,\ldots ,r\}$ are  compact sets \cite{Lasserre_book}. Moreover,
Assumption 2.1 can be easily satisfied when $K$ and $F$ are nonempty compact
sets, and one knows the bounds $N_1$ for $\Vert x\Vert $ on $K$ and
$N_2$ for $\Vert y \Vert$ on $F$.
Indeed, in this case, it suffices to add redundant constraints $%
g_{s+1}(x,y)=\Vert (x,y)\Vert ^{2}-(N_1^{2}+N_2^2)$ and $h_{r+1}(y)=\Vert y\Vert
^{2}-N_2^{2}$ to the definition of $K$ and $F$ respectively, and Assumption
2.1 is satisfied with $R_1=N_1^{2}+N_2^2$, $R_2=N_2^2$,  $\sigma_{i} \equiv 0$ for all $1\le i \le s$, $\bar \sigma_j \equiv 0$ for all $1 \le j \le r$ and $\sigma _{s+1}=\bar \sigma _{r+1}\equiv 1$.
We also note that, under Assumption 2.1, a solution for problem (P) exists by Corollary \ref{cor:1}.

\section{Convex Lower Level Problems}

%In this section, we  consider the polynomial bilevel programming problems where the lower level problem is convex. In other words, consider
%\begin{eqnarray*}
% &\displaystyle \min_{x \in \mathbb{R}^n, y \in \mathbb{R}^m}& f(x,y) \\
%& \mbox{subject to } & g_i(x,y) \le 0, \ i=1,\ldots,s, \\
%& & y \in Y(x):=\displaystyle {\rm argmin}_{w \in \mathbb{R}^m}\{G(x,w): h_j(w) \le 0,j=1,\ldots,r\},
%\end{eqnarray*}
% where for each $x \in \mathbb{R}^n$, $G(x,\cdot)$ is a convex polynomial, $h_j$ are polynomials, $j=1,\ldots,r$, and the feasible set $F=\{w \in \mathbb{R}^m: h_j(w) \le 0\}$ is a convex set.

In this section, we consider the convex polynomial bilevel programming
problem $(P)$ where the lower level problem is convex in the sense that,
for each $x\in \mathbb{R}^{n}$, $G(x,\cdot )$ is a convex polynomial, $h_{j}$
are polynomials, $j=1,\ldots ,r$, and the feasible set of lower level problem $F:=\{w\in \mathbb{R}%
^{m}\ :\ h_{j}(w)\leq 0,j=1,\ldots ,r\}$ is a convex set. We note that, the
representing polynomials $h_{j}$ which describes the convex feasible set $F$
need not to be convex, in general. % In other words, consider
% \begin{eqnarray*}
%  &\displaystyle \min_{x \in \mathbb{R}^n, y \in \mathbb{R}^m}& f(x,y) \\
% & \mbox{subject to } & g_i(x,y) \le 0, \ i=1,\ldots,s, \\
% & & y \in Y(x):=\displaystyle {\rm argmin}_{y \in \mathbb{R}^m}\{G(x,y): h_j(y) \le 0,j=1,\ldots,r\},
% \end{eqnarray*}
%  where for each $x \in \mathbb{R}^n$, $G(x,\cdot)$ is a convex polynomial and $h_j$ are convex polynomials, $j=1,\ldots,r$.

%  $G(x,y)=a_1^T x+a_2^Ty+\alpha$ with $a_1 \in \mathbb{R}^n$,  $a_2 \in \mathbb{R}^m$ and $\alpha \in \mathbb{R}$, $h_j(x,w)=b_{1j}^Tx+b_{2j}^Ty +\beta_j$ with $b_{1j},b_{2j} \in \mathbb{R}^m$ and $\beta_j \in \mathbb{R}$.

We say that the lower level convex problem of $(P)$ satisfies the \textit{%
nondegeneracy condition} if for each $j=1,\ldots ,r$,
\begin{equation*}
y\in F\mbox{ and }h_{j}(y)=0\ \Rightarrow \ \nabla h_{j}(y)\neq 0.
\end{equation*}%
Recall that the lower level convex problem of $(P)$ is said to satisfy the
\textit{Slater condition} whenever there exists $y_{0}\in \mathbb{R}^{m}$
such that $h_{j}(y_{0})<0$, $j=1,\ldots ,r$. Note that, under the Slater
condition, the lower level problem automatically satisfies the nondegeneracy
condition if each $h_{j}$, $j=1,\ldots ,r$ is a convex polynomial.

% Recall that $(x,y,\lambda) \in \mathbb{R}^n \times \mathbb{R}^m \times
% \mathbb{R}^r$ is a Karush-Kuhn-Tucker (KKT) point of the lower level problem
% of (P) whenever
% \begin{equation*}
% \nabla_y G(x,y)+\sum_{j=1}^r \lambda_j \nabla h_j(y)=0, \  \lambda_j
% h_j(y)=0, \, \lambda_j \ge 0, \, h_j(y) \le 0, \, j=1,\ldots,r.
% \end{equation*}

%  $G(x,y)=a_1^T x+a_2^Ty+\alpha$ with $a_1 \in \mathbb{R}^n$,  $a_2 \in \mathbb{R}^m$ and $\alpha \in \mathbb{R}$, $h_j(x,w)=b_{1j}^Tx+b_{2j}^Ty +\beta_j$ with $b_{1j},b_{2j} \in \mathbb{R}^m$ and $\beta_j \in \mathbb{R}$.

Let us recall
a lemma which provides a link between a KKT point and a minimizer for a
convex optimization problem where the representing function of the convex feasible
set is not necessarily convex.

\begin{lemma}
(\cite[Theorem 2.1]{lasOptl}) Let $\phi$ be a convex function on $\mathbb{R}^m$ and $F:=\{w \in
\mathbb{R}^m: h_j(w) \le 0, j=1,\ldots,r\}$ be a convex set. Suppose that both the
nondegeneracy condition and the Slater condition hold. Then, a point $y$
is a global minimizer of $\min\{\phi(w): w \in F\}$ if and only if $y$ is a
KKT point of $\min\{\phi(w): w \in F\}$, in the sense that, there exist $\lambda_j \ge 0$, $j=1,\ldots,r$ such that
\[
\nabla \phi (y)+\sum_{j=1}^r \lambda_j \nabla h_j(y)=0, \  \lambda_j
h_j(y)=0,  \, h_j(y) \le 0, \, j=1,\ldots,r.
\]
\end{lemma}

We see in the following proposition that a polynomial bilevel programming
problem with convex lower level problem can be equivalently rewritten as a
single level polynomial optimization problem in a higher dimension under the
nondegeneracy condition and the Slater condition. In the special case where all the representing
polynomials $h_j$ are convex, this lemma has been established in \cite{Dempe1}.

\begin{proposition}{\bf (Equivalent single-level problem)}
\label{prop:1}
%Let $K=\{x \in \mathbb{R}^n: g_i(x,y) \le 0\}$ and $F=\{y \in \mathbb{R}^m: h_j(y) \le 0\}$ be compact sets.
Consider problem (P) where the lower level problem is convex.
%, that is, $G(x,\cdot)$ is a convex polynomial, $h_j$ are polynomials, $j=1,\ldots,r$, and the feasible set $F=\{w \in \mathbb{R}^m: h_j(w) \le 0\}$ is a convex set.
Suppose that the lower level problem satisfies
both the nondegeneracy condition and the Slater condition. Then, $(x,y) \in
\mathbb{R}^n \times \mathbb{R}^m$ is a global solution of the bilevel
polynomial optimization problem (P) if and only if there exist Lagrange multipliers\footnote{Indeed, as shown in the proof, $\lambda_0 \neq 0$ always holds under our assumptions. See Remark \ref{remark:2} for a detailed discussion.} $%
\lambda=(\lambda_0,\ldots,\lambda_r) \in \mathbb{R}^{r+1}$  such that $%
(x,y,\lambda) \in \mathbb{R}^n \times \mathbb{R}^m \times \mathbb{R}^{r+1}$
is a global solution of the following single level polynomial optimization
problem:
\begin{eqnarray}  \label{eq:77}
(\overline{P}) &\displaystyle \min_{x \in \mathbb{R}^n, y \in \mathbb{R}%
^m,\lambda \in \mathbb{R}^{r+1}}& f(x,y)  \notag \\
& \mbox{subject to } & g_i(x,y) \le 0, \ i=1,\ldots,s, \\
& & \lambda_0 \nabla_y G(x,y)+\sum_{j=1}^r \lambda_j \nabla h_j(y)=0,  \notag
\\
& & \lambda_0 \ge 0,\, \sum_{j=0}^r \lambda_j^2 =1,\, \lambda_j h_j(y)=0, \,
\lambda_j \ge 0, \, h_j(y) \le 0, \, j=1,\ldots,r.  \notag
\end{eqnarray}
%where $M$ is a positive number such that $M>M_0:=\left(\frac{\min_{x \in K}G(x,y_0)-\min_{(x,y) \in K \times F}G(x,y)}{-\max_{1 \le j \le %r}\{h_j(y_0)\}}\right)^2.$
\end{proposition}

\begin{proof}
Fix any $x \in \mathbb{R}^n$.  %As $\{y: h_j(x,y) \ge 0\}$ is compact, $G^*:=\min_{y \in \mathbb{R}^m}\{G(x,y): h_j(x,y) \ge 0,j=1,\ldots,r\} \in \mathbb{R}$.
% From our assumption and the Lagrangian duality, $y \in Y(x)$ if and only if $h_j(y) \le 0$ and there exist $\lambda_j \ge 0$, $j=1,\ldots,r$, such that
%\[
%G(x,w)-G(x, y)+\sum_{j=1}^r \lambda_j  h_j(w) \ge 0 \mbox{ for all } w \in \mathbb{R}^m
%\]
%Due to the convexity of $G(x,\cdot)$ and $h_j$, $y \in Y(x)$  is equivalent to there exist $\mu_j \ge 0$, $j=1,\ldots,r$, such that
%\begin{eqnarray}\label{eq:009}
%& & \nabla_y G(x,y)+\sum_{j=1}^r \mu_j \nabla h_j(y)=0, \nonumber \\
%& & \mu_j h_j(y)=0, \, \mu_j \ge 0, \, h_j(y) \le 0, \, j=1,\ldots,r.
%\end{eqnarray}
The conclusion will follow if we show that $y \in Y(x)$  is equivalent to the condition that there exist $\lambda_j \ge 0$, $j=0,1,\ldots,r$ such that
\begin{eqnarray}\label{eq:0098}
& & \lambda_0 \nabla_y G(x,y)+\sum_{j=1}^r \lambda_j \nabla h_j(y)=0, \nonumber  \\
%& & \lambda_j h_j(y)=0, \, \lambda_j \ge 0, \, h_j(y) \le 0, \, j=1,\ldots,r.
& & \, \lambda_j h_j(y)=0, \, \lambda_j \ge 0, \, h_j(y) \le 0, \, j=1,\ldots,r,  \\
& & \lambda_0 \ge 0, \sum_{j=0}^r \lambda_j^2 =1. \nonumber
\end{eqnarray}
To see the equivalence, we first assume that $y \in Y(x)$.  Under both the nondegeneracy condition and the Slater condition,
the preceding Lemma
 guarantees that  there exist $\mu_j \ge 0$, $j=1,\ldots,r$, such that
\begin{eqnarray}\label{eq:009}
\nabla_y G(x,y)+\sum_{j=1}^r \mu_j \nabla h_j(y)=0, \
\mu_j h_j(y)=0, \, \mu_j \ge 0, \, h_j(y) \le 0, \, j=1,\ldots,r.
\end{eqnarray}
So, (\ref{eq:0098}) holds with $\lambda_0=\frac{1}{\sqrt{1+\sum_{j=1}^r \mu_j^2}}$ and $\lambda_j=\frac{\mu_j}{\sqrt{1+\sum_{j=1}^r \mu_j^2}}$, $j=1,\ldots,r$.

Conversely, let $(x,y,\lambda)$ satisfy (\ref{eq:0098}). %Then, there exists $z \in \mathbb{R}^n$ with
%\[
%                                                                                      a_j^Tx+b_j^Ty+c_j^Tz+r_j <0 , \ j=1,\ldots,q.
%\]
We now show that $\lambda_0 \neq 0$. Indeed, assume on the contrary that $\lambda_0=0$. Then, $\sum_{j=1}^r \lambda_j^2=1$, $\sum_{j=1}^r \lambda_j \nabla h_j(y)=0$, $\lambda_j h_j(y)=0$,  $\lambda_j \ge 0$ and $h_j(y) \le 0$  $j=1,\ldots,r$.
 Let $J=\{j \in \{1,\ldots,r\}: \lambda_j >0\} \neq \emptyset$. From the Slater condition,
there exists $y_0 \in \mathbb{R}^m$ such that $h_j(y_0)<0$, $j=1,\ldots,r$. Then, there exists $\rho>0$ such that $h_j(w)<0$ for all $w \in \mathbb{R}^m$ with $\|w-y_0\|\le \rho$. As $\sum_{j=1}^r \lambda_j \nabla h_j(y)=0$, we obtain
\begin{equation}\label{eq:oo0}
\sum_{j \in J}\lambda_j \nabla h_j(y)^T(w-y)=0 \mbox{ for all } w \mbox{ with } \|w-y_0\|\le \rho.
\end{equation}
We now see that $\nabla h_j(y)^T(w-y) \le 0$ for all $w$ with $\|w-y_0\|\le \rho$ and for all $j \in J$. (Suppose on the contrary that there exists $w_0$ with $\|w_0-y_0\|\le \rho$
and $j_0 \in J$ such that $\nabla h_{j_0}(y)^T(w_0-y)>0$. By continuity, for all small $t$, $h_{j_0}(y+t(w_0-y))>0$, and hence $y+t(w_0-y) \notin F$.  On the other hand, from our choice of $\rho$, we see that $h_j(w_0)<0$ for all $j=1,\ldots,r$. So, $w_0 \in F$. It then follows from the convexity of $F$ that
$y+t(w_0-y) \in F$ for all small $t$. This is impossible.) This together with (\ref{eq:oo0}) and $\lambda_j=0$ for all $j \notin J$ shows that
\[
\nabla h_j(y)^T(w-y)=0 \mbox{ for all } w \mbox{ with } \|w-y_0\|\le \rho \mbox{ and } j \in J,
\]
and so, $\nabla h_j(y)=0$ for all $j \in J$. Note that $y \in F$ and $h_j(y)=0$ for all $j \in J$. This contradicts the non-degeneracy condition, and so, $\lambda_0 \neq 0$.
%
%We now see that $y \in Y(x)$. If not, then there exists $w \in F$ such that $G(x,w)<G(x,y)$. So, we have
%\begin{eqnarray*}
%0&>&  G(x,y)-G(x,w) \\
% & \ge & \nabla_y G(x,w)^T(y-w) \\
% & = &
%\end{eqnarray*}
%
% As $\sum_{j=1}^r \lambda_j  h_j$ is a convex function, this implies that
%\[
%\sum_{j=1}^r \lambda_j  h_j(w) \ge  \sum_{j=1}^r \lambda_j  h_j(y)=0 \mbox{ for all } w \in \mathbb{R}^n.
%\]
%Letting $w=x_0$, then we have
%\[
%(\sum_{j=1}^r \lambda_j) \max_{1 \le j \le r} \{h_j(x_0)\}=\sum_{j=1}^r \lambda_j  h_j(x_0) \ge 0,
%\]
%which is impossible as $h_j(x_0)<0$,  $\lambda_j \ge 0$, $j=1,\ldots,r$, and $\sum_{j=1}^r \lambda_j=1$.
Thus, by dividing $\lambda_0$ on both sides of the first relation of (\ref{eq:0098}), we see that (\ref{eq:009}) holds. This shows that $y \in Y(x)$ by the preceding lemma again.
%Therefore, to see the conclusion, it suffices to justify that $\displaystyle \sum_{j=1}^r \lambda_j^2 \le M_0$. To see this, note that, for each $x \in K$ and $y \in Y(x)$, we have
%\[
%G(x,w)-G(x, y)+\sum_{j=1}^r \lambda_j  h_j(w) \ge 0 \mbox{ for all } w \in \mathbb{R}^m.
%\]
%Letting $w=y_0$, then,
%\begin{eqnarray*}
%  0 & \le & \inf_{x \in K}\{G(x,y_0)-G(x, y)+\sum_{j=1}^r \lambda_j  h_j(y_0)\}  \\
%  & \le & \min_{x \in K}G(x,y_0)- \min_{x \in K, y \in F}G(x, y)+(\sum_{j=1}^r \lambda_j)  \max_{1 \le j \le r}\{h_j(y_0)\}.
%\end{eqnarray*}
%This shows that $\sum_{j=1}^r \lambda_j \le \frac{\min_{x \in K}G(x,y_0)-\min_{(x,y) \in K \times F}G(x,y)}{-\max_{1 \le j \le r}\{h_j(x_0)\}}$. Note that $\lambda_j \ge 0$. It then follows that
%\[
%\sum_{j=1}^r \lambda_j^2 \le (\sum_{j=1}^r \lambda_j)^2 \le M_0.
%\]
\end{proof}
\begin{remark}\label{remark:2}{\bf (Importance of nondegeneracy and Slater's conditions)}
In Proposition \ref{prop:1}, we require that the nondegeneracy condition and the Slater condition hold. These assumptions provide us a simple uniform bound
for the multipliers $\lambda_0,\ldots,\lambda_r$ in the lower level problem which plays an important role in our convergence analysis later in Theorem \ref{th:90}. Indeed, these assumptions ensure that $\lambda_0 \neq 0$, and so, in particular,
 the equivalence of the following two systems:
 \begin{equation}
{\small \left[\begin{array}{c}
 \displaystyle \lambda_0 \nabla_y G(x,y)+\sum_{j=1}^r \lambda_j \nabla h_j(y)=0,  \notag
\\
\displaystyle \lambda_0 \ge 0,\, \lambda_j h_j(y)=0, \,
\lambda_j \ge 0, \, h_j(y) \le 0, \, j=1,\ldots,r.  \notag \\
\displaystyle \sum_{j=0}^r \lambda_j^2 =1, \notag
       \end{array}
 \right]  \ \Leftrightarrow \   \left[\begin{array}{c}
 \displaystyle \nabla_y G(x,y)+\sum_{j=1}^r \mu_j \nabla h_j(y)=0,  \notag
\\
 \displaystyle \mu_j h_j(y)=0, \,
\mu_j \ge 0, \, h_j(y) \le 0, \, j=1,\ldots,r.  \notag
       \end{array}
 \right]}
 \end{equation}
\end{remark}

Note that the non-degeneracy condition is satisfied when the representing
functions $h_j$, $j=1,\ldots,r$, are convex polynomials and the Slater condition holds. Thus, in this special case, the Slater condition alone
is enough for transforming the polynomial bilevel problem with a
convex lower level problem to a single-level polynomial optimization
problem.

The following simple example illustrates that the preceding
Proposition can be applied to the case where $h_j$'s need not be convex polynomials.
\begin{example}Consider the bilevel problem
\begin{eqnarray*}
(EP_1) &\displaystyle \min_{x \in \mathbb{R}, y \in \mathbb{R}^2}& -x^6+y_1^2+y_2^2 \\
& \mbox{subject to } & x^2+y_1^2+y_2^2 \le 2 \\
& & y \in Y(x):=\displaystyle {\rm argmin}_{w \in \mathbb{R}^2}\{x (w_1+w_2): 1-w_1w_2 \le 0, 0 \le w_1 \le 1, 0 \le w_2 \le 1 \}.
\end{eqnarray*}
Clearly, the lower level problem of $(EP_1)$ is convex but the polynomial $(w_1,w_2) \mapsto 1-w_1w_2$ is not convex. It can be verified that the non-degeneracy condition and Slater condition hold, and so, $(EP_1)$ is equivalent to the following single level polynomial optimization problem
 \begin{eqnarray*}
 &\displaystyle \min_{x \in \mathbb{R}, y \in \mathbb{R}^2,(\lambda_0,\ldots,\lambda_5) \in \mathbb{R}^6}& -x^6+y_1^2+y_2^2 \\
& \mbox{subject to } & x^2+y_1^2+y_2^2 \le 2 \\
& & 1-y_1y_2 \le 0, 0 \le y_1 \le 1, 0\le y_2 \le 1 \\
& & \lambda_0 x+\lambda_1(-y_2)-\lambda_2+\lambda_3=0 \\
& & \lambda_0 x+\lambda_1(-y_1)-\lambda_4+\lambda_5=0 \\
& & \lambda_1(1-y_1y_2)=0,\, \lambda_2 y_1=0, \, \lambda_3 (1-y_1)=0 \\
& & \lambda_4 y_2=0, \  \lambda_5 (1-y_2)=0 \\
& & \lambda_j \ge 0, j=0,1,\ldots,5, \sum_{j=0}^5 \lambda_j^2=1.
%& & y \in Y(x):=\displaystyle {\rm argmin}_{w \in \mathbb{R}^2}\{x^2+x_2^2+w^2: 1-w_1w_2 \le 0, w_1 \ge 0, w_2 \ge 0 \}.
\end{eqnarray*}
\end{example}

Proposition \ref{prop:1} enables us to construct a sequence of semidefinite
programming problems for solving a polynomial bilevel programming problem
with a convex lower level problem. To do this,  we denote
\begin{eqnarray*}
\widehat{G}_p(x,y,\lambda)=\left\{%
\begin{array}{ll}
g_p(x,y) & p=1,\ldots,s, \\
h_{p-s}(y) & p=s+1,\ldots,s+r, \\
-\lambda_{p-(s+r+1)} & p=s+r+1,\ldots , s+2r+1,%
\end{array}
\right.
\end{eqnarray*}
and
\begin{eqnarray*}
\widehat{H}_q(x,y,\lambda)=\left\{%
\begin{array}{ll}
\lambda_q h_q(y), & q=1,\ldots,r \\
\displaystyle \left(\lambda_0 \nabla_y G(x,y)+\sum_{j=1}^r \lambda_j \nabla h_j(y)\right)_{q-r},
& q=r+1, \ldots,r+m \\
\displaystyle \sum_{j=0}^r \lambda_j^2-1, & q=r+m+1,%
\end{array}%
\right.
\end{eqnarray*}
where $\left(\lambda_0 \nabla_y G(x,y)+\sum_{j=1}^r \lambda_j \nabla h_j(y)\right)_{i}$ is the $i$th coordinate of $\lambda_0 \nabla_y G(x,y)+\sum_{j=1}^r \lambda_j \nabla h_j(y)$, $i=1,\ldots,m$.
We also denote the degree of $\widehat{G}_p$ to be $u_p$ and the degree of $%
\widehat{H}_q$ to be $v_q$ .

We now introduce a sequence of sums-of-squares relaxation problems as follows:
for each $k \in \mathbb{N}$,
\begin{eqnarray}  \label{eq:3.6}
(Q_k) &\max_{\mu, \sigma_p} & \mu \\
&\mbox{ s.t. } & f- \mu =\sigma_0-\sum_{p=1}^{s+2r+1} \sigma_p \widehat{G}%
_p-\sum_{q=1}^{r+m+1} \phi_q \widehat{H}_q  \notag \\
& & \sigma_p \in \Sigma^2[x,y, \lambda], \ p=0,1,\ldots,s+2r+1,  \notag \\
& & \mathrm{deg} \sigma_0 \le 2k, \ \mathrm{deg} (\sigma_p \widehat{G}_p)
\le 2k, p=1,\ldots,s+2r+1,  \notag \\
& & \phi_q \in \mathbb{R}[x,y,\lambda], \, q=1,\ldots,r+m+1, \ \mathrm{deg}
(\phi_q \widehat{H}_q) \le 2k, q=1,\ldots,r+m+1.  \notag
\end{eqnarray}
It is  known that each $(Q_k)$ can be reformulated as a semidefinite programming problem  \cite{Lasserre_book}.

\begin{theorem}\label{th:90}
\textbf{(Convex lower level problem: Convergence theorem)}
%Let $K=\{x \in \mathbb{R}^n: g_i(x,y) \le 0\}$ and $F=\{y \in \mathbb{R}^m: h_j(y) \le 0\}$ be compact sets.
Consider the problem (P) where the lower level problem is convex.
%, that is, $G(x,\cdot)$ is a convex polynomial, $h_j$ are polynomials, $j=1,\ldots,r$, and the feasible set $F=\{w \in \mathbb{R}^m: h_j(w) \le 0\}$ is a convex set.
Suppose that Assumption 2.1 holds and that the lower level problem satisfies
both the nondegeneracy condition and the Slater condition. Then,  ${\rm val}(Q_{k}) \le {\rm val}(Q_{k+1})$ for all $k \in \mathbb{N}$ and $%
\mathrm{val}(Q_{k})\rightarrow \mathrm{val}(P)$ as $k \rightarrow \infty$, where ${\rm val}(Q_k)$ and ${\rm val}(P)$ denote the optimal value of the problems $(Q_k)$ and (P) respectively. %and $\mathrm{val}%
%(Q_{k}^{\ast })\rightarrow \mathrm{val}(P)$.
% Moreover, suppose that problem
% (P) has a unique solution say $(\bar x, \bar y)$ and let $\mathbf{z}_{k}$ be $(1/k)$%
% -solution of $(Q_{k}^{\ast })$. Then, as $%
% k\rightarrow \infty$, we have $%
% (L_{\mathbf{z}_k}(X_1),\ldots,L_{\mathbf{z}_k}(X_n)) \rightarrow \bar x$,  and $%
% (L_{\mathbf{z}_k}(X_{m+1}),\ldots,L_{\mathbf{z}_k}(X_{m+n})) \rightarrow \bar y$, where $X_i$ denotes the polynomial which maps each vector to its $i$th coordinate,  $i=1,\ldots,m+n$.
\end{theorem}

\begin{proof} From Corollary \ref{cor:1}, a global solution of (P) exists. Let $(x,y)$ be a global solution of $(P)$.  From Proposition \ref{prop:1}, there exists $\lambda \in \mathbb{R}^{r+1}$ such that $(x,y,\lambda)$ is a solution of $(\overline{P})$
and ${\rm val}(P)={\rm val}(\overline{P})$.

From the construction of $(Q_k)$, $k \in \mathbb{N}$, it can be easily verified that ${\rm val}(Q_{k}) \le {\rm val}(Q_{k+1}) \le {\rm val}(\overline{P})$ for all $k \in \mathbb{N}$. Let $\epsilon>0$.  Define $\hat{f}(x,y,\lambda)=f(x,y)- ({\rm val}(\bar P)-\epsilon)$.  Note that the feasible set $U$ of $(\overline{P})$ can be written as
\begin{eqnarray*}
U=\{(x,y,\lambda) \in \mathbb{R}^n \times \mathbb{R}^m \times \mathbb{R}^{r+1}& : &  -\hat{G}_p(x,y,\lambda) \ge 0, \ p=1,\ldots,s+2r+1, \\
& & -\hat{H}_q(x,y,\lambda) \ge 0, \hat{H}_q(x,y,\lambda) \ge 0,  \ q=1,\ldots,r+m+1\}.
\end{eqnarray*}
Then, we see that $\hat{f}>0$ over $U$.
We now verify that the conditions in  Putinar's Positivstellensatz (Lemma \ref{putinar})  are satisfied. To see this, from Assumption 2.1, there exist $R_1,R_2>0$ such that
\[
R_1-\Vert (x,y)\Vert ^{2}=\sigma _{0}(x,y)-\sum_{i=1}^{s}\sigma _{i}(x,y)g_i(x,y) \ \
\mbox{ and }  \ \
R_2-\Vert y \Vert ^{2}= \bar \sigma_{0}(y)-\sum_{j=1}^{r}\bar \sigma
_{j}(y)h_{j}(y),
\]
for some sums-of-squares polynomials $\sigma _{0},\sigma _{1},\ldots ,\sigma
_{s}\in \Sigma^2 [x,y]$ and sums-of-squares polynomials $\bar \sigma_0,\bar \sigma_1, \ldots, \bar \sigma_r \in \Sigma^2[y]$.
Letting $\lambda=(\lambda_0,\lambda_1,\ldots,\lambda_r) \in \mathbb{R}^{r+1}$, we obtain that
\begin{eqnarray*}
(1+R_1+R_2)-\|(x,y,\lambda)\|^2& =& \left(\sigma _{0}(x,y)+ \bar \sigma_{0}(y)\right)-\sum_{j=1}^r \bar \sigma_j(y) h_j(y)-\sum_{i=1}^s \sigma_{i}(x,y) g_i(x,y)+ (1- \sum_{j=0}^r \lambda_j^2) \\
& = & \left(\sigma _{0}(x,y)+ \bar \sigma_{0}(y)\right)-\sum_{j=1}^r \bar \sigma_j(y) \hat{G}_{s+j}(x,y,\lambda)\\
& & -\sum_{i=1}^s \sigma_{i}(x,y) \hat{G}_i(x,y)- \hat{H}_{r+m+1}(x,y,\lambda).
\end{eqnarray*}
%  Then, we see that $$(1+R_1+R_2)-\|(x,y,\lambda)\|^2 \in {\bf M}(-\hat{G}_1,\ldots,-\hat{G}_{s+2r+1},-\hat{H}_1,\ldots,-\hat{H}_{r+m+1},\hat{H}_1,\ldots,\hat{H}_{r+m+1}),$$
%  where the set ${\bf M}$ is defined as in (\ref{eq:M}).
% Moreover, $\{(x,y,\lambda): (1+R_1+R_2)-\|(x,y,\lambda)\|^2 \ge 0\}$ is a compact set.
So, applying Putinar's Positivstellensatz (Lemma \ref{putinar}) with
$w=(x,y,\lambda) \in \mathbb{R}^{m}\times \mathbb{R}^n \times \mathbb{R}^{r+1}$,  there exist sums of squares polynomials $\sigma_p \in \Sigma^2[x,y,\lambda]$, $p=0,1,\ldots,s+2r+1$ and sums-of-squares polynomials $\phi_{1q}, \phi_{2q} \in \Sigma^2[x,y,\lambda]$, $q=1,\ldots,r+m+1$ such that
\[
\hat{f}=\sigma_0-\sum_{p=1}^{s+2r+1} \sigma_p \hat{G}_p-\sum_{q=1}^{r+m+1} \phi_{1q} \hat{H}_q+\sum_{q=1}^{r+m+1} \phi_{2q} \hat{H}_q.
\]
Let $\phi_q \in \mathbb{R}[x,y,\lambda]$ be a real polynomial defined by $\phi_q=\phi_{1q}-\phi_{2q}$, $q=1,\ldots,r+m+1$. Then, we have
\[
f- ({\rm val}(\overline{P})-\epsilon) =\sigma_0-\sum_{p=1}^{s+2r+1} \sigma_p \hat{G}_p-\sum_{q=1}^{r+m+1} \phi_{q} \hat{H}_q.
\]
Thus, there exists $k \in \mathbb{N}$, ${\rm val}(Q_k) \ge {\rm val}(\overline{P})-\epsilon={\rm val}(P)-\epsilon$. Note that, by the construction, ${\rm val}(Q_k) \le {\rm val}(\overline{P})={\rm val}(P)$ for all $k \in \mathbb{N}$. Therefore, ${\rm val}(Q_k) \rightarrow {\rm val}(\overline{P})={\rm val}(P)$.
\end{proof}

\begin{remark}{\bf (Convergence to a global minimizer)}
 It is worth noting that, in addition to the assumptions of Theorem 3.5, if we further assume that the equivalent problem
$(\overline{P})$ has a unique solution say $(\bar x, \bar y)$, then we can also find the global minimizer $(\bar x,\bar y)$ with the help of the above sequential SDP relaxation problems.
In fact, as each $(Q_k)$ is a semidefinite programming problem,
its corresponding dual problem (see \cite{Lasserre_book}) can be formulated as
\begin{eqnarray*}
(Q_{k}^*) &\displaystyle \inf_{\mathbf{z} \in \mathbb{N}^{n+m+r}_{2k}}& L_{%
\mathbf{z}}(f) \\
& \mbox{subject to } & \mathbf{M}_{k}(\mathbf{z}) \succeq 0, \ \mathbf{z}_{\bf 0}=1,  \\
& & \mathbf{M}_{k- \overline{u_p}}(\widehat{G}_p,\mathbf{z}) \succeq 0, p=1,\ldots,s+2r+1
\\
& & \mathbf{M}_{k-\overline{v_q}}(\widehat{H}_q, \mathbf{z}) = 0, q=1,\ldots,r+m+1,
\end{eqnarray*} 
where $\overline{u_p}$ (resp. $\overline{v_q}$) is the largest integer which is smaller than $\frac{u_p}{2}$ (resp. $\frac{v_q}{2}$), $L_{\mathbf{z}}$ is the Riesz functional defined by $L_{\mathbf{z}%
}(f)=\sum_{\alpha} f_{\alpha} z_{\alpha}$ with $f(x)=\sum_{\alpha}
f_{\alpha}x^{\alpha}$ and, for a polynomial $f$, $\mathbf{M}_t (f, z)$, $t
\in \mathbb{N}$ is the so-called localization matrix defined by
%\begin{equation*}
$\left[\mathbf{M}_t (f, \mathbf{z}) \right]_{\alpha, \beta}=\sum_{\gamma}
f_{\gamma} z_{\alpha+\beta+\gamma}$ for all $\alpha,\beta \in \mathbb{N}^{n+m+r}_t$.
%\end{equation*}
From the weak duality, one has ${\rm val}(\overline{P}) \ge {\rm val}(Q_k^*) \ge {\rm val}(Q_k)$. Thus, the preceding theorem together with
${\rm val}(P)={\rm val}(\overline{P})$ implies that ${\rm val}(Q_k^*) \rightarrow {\rm val}(P)$. Moreover, it was shown in \cite[Theorem 4.2]{Lasserre2000} that if the feasible set of the polynomial optimization problem $(\bar P)$ has a non-empty interior, then there exists a natural number $N_0$ such that ${\rm val}(Q_k^*)= {\rm val}(Q_k)$ for all $k \ge N_0.$

%Moreover,
Let $\mathbf{z}_{k}$ be a %
solution of $(Q_{k}^{\ast })$. Then, as $%
k\rightarrow \infty$, we have $%
(L_{\mathbf{z}_k}(X_1),\ldots,L_{\mathbf{z}_k}(X_n)) \rightarrow \bar x$,  and $%
(L_{\mathbf{z}_k}(X_{n+1}),\ldots,L_{\mathbf{z}_k}(X_{n+m})) \rightarrow \bar y$, where $X_i$ denotes the polynomial which maps each vector to its $i$th coordinate,  $i=1,\ldots,n+m$. The conclusion follows from \cite{Sch}.
\end{remark}

%In this case, for all $x \in \mathbb{R}^n$, Farkas' lemma implies that
%$y \in Y(x):=\displaystyle {\rm argmin}_{w \in \mathbb{R}^n}\{G(x,w): h_j(w) \ge 0,j=1,\ldots,r\}$ is equivalent to the fact that there exist $\lambda_j \ge 0$ such that
%\[
%a_2=\sum_{j=1}^r\lambda_j b_{2j}  \mbox{ and } - a_2^Ty- \sum_{j=1}^r\lambda_j (b_{1j}^Tx+\beta_j) \ge 0 \ \mbox{(This is a quadratic inequality as $\lambda_j$ and $x$ are variables)}.
%\]
%Then, the above bilevel problem can be converted as a single polynomial optimization problem with additional variables
%\begin{eqnarray*}
% &\displaystyle \min_{x \in \mathbb{R}^n, y \in \mathbb{R}^m,\lambda \in \mathbb{R}^r}& f(x,y) \\
%& \mbox{subject to } & g_i(x,y) \ge 0, \ i=1,\ldots,s, \\
%& & a_2=\sum_{j=1}^r\lambda_j b_{2j}, - a_2^Ty- \sum_{j=1}^r\lambda_j (b_{1j}^Tx+\beta_j) \ge 0, \lambda \in \mathbb{R}^r_+
%\end{eqnarray*}

The above theorem shows that one can use a sequence of semidefinite
programming problems to approximate the global optimal value of a bilevel polynomial
optimization problem with convex lower level problem. Moreover, under a sufficient rank condition (see \cite[Theorem 5.5]{Lasserre_book}), one can check whether finite
convergence has occurred, i.e., by testing whether ${\rm val}(Q_{k_0})={\rm val}(\overline{P})$ for some $k_0 \in \mathbb{N}$. This rank
condition has been implemented in the software GloptiPoly 3 \cite{Gloptpoly} along with a linear algebra procedure to extract global minimizers of a polynomial optimization problem.

We now
provide a simple example to illustrate how to use sequential SDP relaxations to solve the
bilevel polynomial optimization problems with convex lower level problem:
\begin{example}{\bf (Solution by sequential SDP relaxations)}
Consider the following simple bilevel polynomial optimization problem
\begin{eqnarray*}
& \displaystyle \min_{(x,y)\in \mathbb{R}^2} & xy^5-y^6 \\
& & x^2+y^2 \le 2 \\
& & y \in Y(x):={\rm argmin}_{w \in \mathbb{R}}\{xw: -1 \le w \le 1\}.
\end{eqnarray*}
Direct verification shows that there are two global solutions $(-1,1)$ and $(1,-1)$ with global optimal value $2$.
We note that the lower level problem is convex and it is equivalent to the following single level polynomial optimization problem
\begin{eqnarray*}
& \displaystyle \min_{(x,y,\lambda_0,\lambda_1,\lambda_2)\in \mathbb{R}^5} & xy^5-y^6 \\
& & x^2+y^2 \le 2 \\
& & \lambda_0 x+\lambda_1-\lambda_2 =0 \\
& & \lambda_i \ge 0, \lambda_1(y-1)=0, \lambda_2(-1-y)=0, -1 \le y\le 1\\
& & \lambda_0^2+\lambda_1^2+\lambda_2^2 =1.
\end{eqnarray*}
Solving the converted single level polynomial optimization problem using GloptiPoly 3, %(see \cite{Gloptpoly,Lasserre_book}),Gloptpoly 3
% \begin{verbatim}
% mpol x 5
% g0 = x(2)^5*(x(1)-x(2));
% F = [2-x(1)*x(1)-x(2)*x(2)>=0; x(1)*x(3)+x(4)-x(5)==0;
% x(4)*(x(2)-1)==0; x(5)*(-1-x(2))==0;
% x(3)>=0; x(4)>=0; x(5)>=0; -1<=x(2); x(2)<=1; x(3)^2+x(4)^2+x(5)^2==1]
% P=msdp(min(g0),F,4)
% [status,obj]=msol(P)
% \end{verbatim}
the solver extracted two global solutions $(x,y,\lambda_0,\lambda_1,\lambda_2)=(-1.000,1.000,0.7071,0.7071,0)$ and $(x,y,\lambda_0,\lambda_1,\lambda_2)=(1.000,-1.000,0.7071,0,0.7071)$ with the true global optimal value $-2$. %global optimality certified numerically, and the true global optimal value $-2$.
\end{example}

\begin{remark}{\bf (Single level polynomial problem)}
In the case where $(P)$ is a single level problem, Theorem \ref{th:90} yields the known convergence result of the sequential SDP relaxation scheme (often referred to as the Lasserre hierarchy) for solving single level polynomial optimization problems \cite{Lasserre_book}.
 Indeed, consider a (single level) polynomial optimization problem $$(P_0) \ \ \ \min_{x \in \mathbb{R}^n}\{f(x): g_i(x) \le 0, i=1,\ldots,s\}.$$ Suppose that there exist $R>0$ and sums of squares polynomial $\sigma_i \in \Sigma^2[x]$ such that $$R-\Vert x\Vert ^{2}=\sigma_{0}(x)-\sum_{i=1}^{s}\sigma_{i}(x)g_i(x).$$
Let $\hat{f}(x,y)=f(x)$, $\hat{g}_i(x,y)=g_i(x)$, $i=1,\ldots,s$ and $G(x,y)\equiv 0$ for all $(x,y) \in \mathbb{R}^n \times \mathbb{R}$.
We note that ${\rm val}(P_0)$ equals the optimal value of the following bilevel polynomial optimization problem
\begin{eqnarray*}
 &\displaystyle\min_{x\in \mathbb{R}^{n},y\in \mathbb{R}^{m}}&\hat{f}(x,y) \\
&\mbox{subject to }& \hat{g}_i(x,y)\leq 0,\ i=1,\ldots ,s, \\
&&y\in Y(x):=\displaystyle\mathrm{argmin}_{w\in \mathbb{R}%
^{m}}\{0:w^2 \le 1\}
\end{eqnarray*}
Then, Theorem \ref{th:90} yields that
${\rm val}(P_0)=\displaystyle \lim_{k \rightarrow \infty}{\rm val}(Q_{k}^0)$, where, for each $k$, the problems $(Q_k^0)$ is given by
\begin{eqnarray*}
(Q_k^0) &\max_{\mu, \sigma_p} & \mu \notag \\
&\mbox{ s.t. } & f- \mu =\sigma_0-\sum_{p=1}^{s} \sigma_p g_p, \\
& & \sigma_p \in \Sigma^2[x], \ p=0,1,\ldots,s, \, \mathrm{deg} \sigma_0 \le 2k, \ \mathrm{deg} (\sigma_p g_p)
\le 2k, p=1,\ldots,s.  \notag
\end{eqnarray*}
%\end{corollary}
% \begin{proof}
% Let $\epsilon>0$. It can be easily verified that Assumption 2.1 holds, and so,  Theorem \ref{th:90} implies that  there exist sums-of-squares polynomial $\sigma_0,\ldots, \sigma_{s+3} \in \Sigma^2[x,y,\lambda]$ and real polynomials $\phi_1,\phi_2,\phi_3 \in \mathbb{R}[x,y,\lambda]$ such that for all $(x,y,\lambda) \in \mathbb{R}^n \times \mathbb{R} \times \mathbb{R}^2$,
% \begin{eqnarray*}
% f(x)-({\rm val}(P_0)-\epsilon)
% & = & \sigma_0(x,y,\lambda)-\sum_{p=1}^{s} \sigma_p(x,y,\lambda) g_p(x)-\sigma_{s+1}(x,y,\lambda)(-\lambda_0)-\sigma_{s+2}(x,y,\lambda)(-\lambda_1) \\
% & & - \phi_1(x,y,\lambda)\big(\lambda_1 (y^2-1)\big)-\phi_2(x,y,\lambda)(2\lambda_1 y)-\phi_3(x,y,\lambda)(\lambda_0^2+\lambda_1^2-1).
% \end{eqnarray*}
% Letting $y=1$, $\lambda_0=1$ and $\lambda_1=0$, we obtain that
% \[
% f(x)-({\rm val}(P_0)-\epsilon)=\big(\sigma_0(x,1,1,0)+\sigma_{s+1}(x,1,1,0)\big)-\sum_{p=1}^{s} \sigma_p(x,1,1,0) g_p(x).
% \]
% This implies that ${\rm val}(P_0)-\epsilon \le {\rm val}(Q_k^0)$. From the construction, it can be verified that ${\rm val}(Q_k^0) \le {\rm val}(P_0)$ for all $k \in \mathbb{N}$. So, the conclusion follows.
% \end{proof}
\end{remark}
\setcounter{equation}{0}

\section{Nonconvex Lower Level Problems}

In this section, we examine how to solve a bilevel polynomial optimization problem with a nonconvex lower level problem
towards a global minimizer using semi-definite programming hierarchies.
Consider an $\epsilon$-approximation of the general bilevel
polynomial problem $(P)$:
\begin{eqnarray*}
(P_{\epsilon}) &\displaystyle \min_{(x,y) \in \mathbb{R}^n \times \mathbb{R}%
^m}& f(x,y) \\
& \mbox{subject to } & g_i(x,y) \le 0, \ i=1,\ldots,s, \\
& & h_j(y) \le 0, \ j=1,\ldots,r, \\
& & G(x,y)-\min_{w \in \mathbb{R}^m}\{G(x,w): h_j(w) \le 0,j=1,\ldots,r\}
\le \epsilon.
\end{eqnarray*}
The above $\epsilon$-approximation problem plays a key role in the so-called
value function approach for finding a stationary point of a bilevel
programming problems, and has been studied and used widely in the literature
(for example see \cite{Ye,Ye1}). The main idea of the value function approach is
to further approximate the (possibly nonsmooth and nonconvex) function $x
\mapsto \min_{w \in \mathbb{R}^m}\{G(x,w): h_j(w) \le 0,j=1,\ldots,r\}$
using smooth functions, and asymptotically solve the problem by using smooth
local optimization techniques (such as projected gradient method (PG) and
sequential quadratic programming problem (SQP) techniques). For instance,
\cite{Ye} use this approach together with the smoothing projected gradient
method to solve the bilevel optimization problem, in the case where $g_i$ depends on $x$ only, $\{x \in \mathbb{R}^n :g_i(x)\le 0\}$ and $%
\{w \in \mathbb{R}^m:h_j(w) \le 0\}$ are convex sets. The algorithm only converges to a stationary
point of the original problem (in a suitable sense).

{\em We now introduce a
general purpose scheme which enables us to solve $(P_{\epsilon})$ towards
global solutions using SDP hierarchies. The proof techniques for the convergence of this scheme (Theorem 4.6) relies on
the joint-marginal method introduced in \cite{Lasserre2010} to approximate a global solution of a parameterized single level polynomial optimization problem. Here, following the approach in \cite{Lasserre2010}, we extend the scheme and its convergence analysis to the bilevel polynomial optimization setting.}

The following known simple lemma shows that the problem $%
(P_{\epsilon})$ indeed approximates the original bilevel polynomial
optimization problem as $\epsilon \rightarrow 0_+$.  To do this, for $\epsilon,\delta \ge 0$,  recall that  $(\bar x,\bar y)$ is called a $\delta$-global solution of $(P_{\epsilon})$ if $(\bar x,\bar y)$ is feasible
for $(P_{\epsilon})$ and $f(\bar x,\bar y) \le {\rm val}(P_{\epsilon})+\delta$ where ${\rm val}(P_{\epsilon})$ is the optimal value of $(P_{\epsilon})$.
%We note that this lemma has been established for example in \cite{Lemma}. For completeness, we provide a proof below.
\begin{lemma}
\textbf{(Approximation lemma cf. \cite{Lemma})} \label{lemma:approximate} Suppose that $K:=\{(x,y) \in \mathbb{R}^n \times \mathbb{R}^m :g_i(x,y)\le 0\}$ and $%
F=\{w \in \mathbb{R}^m:h_j(w) \le 0\}$ are compact.
%Assumption 2.1 holds.
Let $\epsilon _{k}\rightarrow 0_{+}$ and $\delta
_{k}\rightarrow 0_{+}$ as $k\rightarrow \infty$. Let $(\bar{x}_k,\bar{y}_k)$ be an $\delta_{k}$-global solution for $(P_{\epsilon_k})$. Then,
$\{(\bar{x}_k,\bar{y}_k)\}_{k\in \mathbb{N}}$ is
a bounded sequence and any of its cluster point $(\bar{x},\bar{y})$ is a
solution of the bilevel polynomial optimization problem $(P)$.
\end{lemma}

The following lemma explains the analytic property of the function $\epsilon
\mapsto \mathrm{val}(P_{\epsilon})$, and shows that $\mathrm{val}%
(P_{\epsilon})$ converges to $\mathrm{val}(P)$ in the order of $O(\epsilon^{%
\frac{1}{q}})$ as $\epsilon \rightarrow 0_+$ for some $q \in \mathbb{N}_{ >
0}:=\mathbb{N} \backslash \{0\}$. The proof relies on some important properties and
facts on semialgebraic functions/sets and we delay the proof to the Appendix B.

\begin{lemma}
\textbf{(Analytic property \& approximation quality)} \label{lemma:2}
Suppose that Assumption 2.1 holds. Let $I\subseteq \mathbb{R}%
_{+}:=[0,+\infty )$ be a finite interval. For each $\epsilon \in I$, denote
the optimal value of $(P_{\epsilon })$ by $\mathrm{val}(P_{\epsilon })$.
Then,
\begin{itemize}
\item[\textrm{(i)}] the one-dimensional function $\epsilon \mapsto \mathrm{%
val}(P_{\epsilon})$ is a nonincreasing, lower semicontinuous,
right-continuous and semialgebraic function on $I$. In particular, the function $\epsilon \mapsto \mathrm{val}(P_{\epsilon})$
is continuous over $I$ except at finitely many points.

 \item[\textrm{(ii)}] %there exists $\epsilon_0>0$ such that for all $\epsilon
% \in [0,\epsilon_0)$ we have
% \begin{equation*}
% \mathrm{val}(P_{\epsilon}) = \mathrm{val}(P) + \sum_{k \ge 1} c_k \epsilon^{%
% \frac{k}{q}},
% \end{equation*}
% for some $c_k \in \mathbb{R}$ and $q \in \mathbb{N}_{ > 0}.$ In particular,
There exist $q \in \mathbb{N}_{ > 0}$, $\epsilon_0>0$ and $M>0$ such that for all $\epsilon \in
[0,\epsilon_0)$
\begin{equation*}
\mathrm{val}(P_{\epsilon}) \le \mathrm{val}(P) \le \mathrm{val}%
(P_{\epsilon})+ M \epsilon^{\frac{1}{q}}.
\end{equation*}
\end{itemize}
\end{lemma}

Now, we present a simple example to illustrate the above lemma. It also
implies that, in general, the function $\epsilon \mapsto \mathrm{val}%
(P_{\epsilon})$ can be a discontinuous semialgebraic function. % \begin{example}
\begin{example}
%Let $F := \{y \ | \ y^2(y^2 - 1)^2 \le 0 \} = \{0, 1, -1\}$ and $K$ be any subset in $\mathbb{R}.$
%Let $G(x, y) := x^2 + y^2.$
Consider the bilevel polynomial optimization problem
\begin{eqnarray*}
(EP) & \min_{(x,y) \in \mathbb{R}^2} & y \\
& \mbox{ s.t. } & x^2 \le 1,\\
& & y \in {\rm argmin}_{w \in \mathbb{R}}\{x^2+w^2: w^2(w^2 - 1)^2 \le 0\}.
\end{eqnarray*}
Note that $\displaystyle J(x)=\min_{w \in \mathbb{R}}\{x^2+w^2: w^2(w^2 - 1)^2 \le 0\}=x^2$.
Its $\epsilon$-approximation problem is
\begin{eqnarray*}
(EP_{\epsilon}) & \displaystyle \min_{(x,y) \in \mathbb{R}^2} & y \\
& \mbox{ s.t. } & x^2 \le 1,\\
& & y^2(y^2 - 1)^2 \le 0, \; y^2 \le \epsilon.
\end{eqnarray*}
It can be verified that
$${\rm val}(EP_{\epsilon}) =
\begin{cases}
0, & \textrm{ if } 0 \le \epsilon < 1,\\
-1, & \textrm{ if } \epsilon \ge 1.
\end{cases}$$
Therefore the function $\epsilon \mapsto {\rm val}(EP_{\epsilon})$ is nonincreasing, lower semicontinuous, right-continuous and semialgebraic on $[0, +\infty).$ Moreover, it is continuous
on $[0,\epsilon_0]$ for any $\epsilon_0<1$ and it is discontinuous at $1$.
\end{example}

\subsection*{Solving  $\protect\epsilon$-approximation problems via sequential SDP relaxations%
}

Here, we describe how to solve an $\epsilon$-approximation problem via a sequence of SDP relaxation problems.
One of the key steps is to construct a sequence of polynomials to approximate
the optimal value function of the lower level problem $x \mapsto \min_{w \in
\mathbb{R}^m}\{G(x,w): h_j(w) \le 0,j=1,\ldots,r\}$. In general, the optimal
value function of the lower level problem is merely a continuous function.
We now recall a procedure introduced in \cite{Lasserre2010} to approximate
this optimal value function by a sequence of polynomials.

Recall that $K=\{x:g_i(x,y) \le 0, i=1,\ldots,s\}$. We denote $\mathrm{Pr%
}_1 K=\{x \in \mathbb{R}^n: (x,y) \in K \mbox{ for some } y \in \mathbb{R}%
^m\}$. From Assumption 2.1, $K$ is bounded, and so, $\mathrm{Pr}_1 K$ is
also bounded. Let $\mathrm{Pr}_1 K \subseteq \Omega:=\{x \in \mathbb{R}^n: \|x\|_{\infty} \le
M\}$ for some $M>0$. Let $\theta_l(x)=x_l^2-M^2$, $l=1,\ldots,n$. Then $\Omega=\{x: \theta_l(x) \le 0, l=1,\ldots,n\}$.  Let $\varphi$ be a probability Borel measure supported
on $\Omega$ with uniform distribution on $X$. We note that all the moments of $\varphi$ over $\Omega$ denoted by   $%
\gamma=(\gamma_{\beta}),$ $\beta \in \mathbb{N}^n$, defined by
\begin{equation*}
\gamma_\beta:= \int_{\Omega}  {x}^{\beta} d\varphi ({x}), \
\beta \in \mathbb{N}^n,
\end{equation*}
can be easily computed (see \cite{Lasserre2010}).

For each $k \in \mathbb{N}$ with $k \ge k_0:=\max\{\lceil \frac{\mathrm{deg}%
f}{2}\rceil, \lceil \frac{\mathrm{deg}h_j}{2}\rceil\}$, set $\displaystyle \mathbb{N}_{2k}^{n}:=\{(\alpha_1,\ldots,\alpha_n) \in \mathbb{N}^n: \sum_{l=1}^n \alpha_l \le 2k\}$ and consider the
following optimization problem
\begin{eqnarray}  \label{eq:2.2}
&\max_{\lambda, \sigma_0,\ldots,\sigma_{r+n}} & \sum_{\beta \in \mathbb{N}_{2k}^n} \lambda_\beta
\gamma_\beta  \notag \\
&\mbox{ s.t. } & G(x,y)-\sum_{\beta \in \mathbb{N}_{2k}^n} \lambda_{\beta}
 {x}^{\beta}=\sigma_0( {x}, {y})-\sum_{j=1}^{r} \sigma_j(%
 {x}, {y}) h_j( {y})-\sum_{l=1}^n \sigma_{r+l}(x,y) \theta_l(x) \notag\\
& & \sigma_j \in \Sigma[ {x}, {y}], \ j=0,1,\ldots,r+n   \\
& & \mathrm{deg} \sigma_0 \le 2k, \ \mathrm{deg} (\sigma_jh_j) \le 2k,
j=1,\ldots,r, \, \mathrm{deg} (\sigma_{r+l}\theta_l) \le 2k, l=1,\ldots,n, \notag
\end{eqnarray}
which can be reformulated as a semidefinite programming problem \cite{Lasserre2010}.  Then, for any feasible solution $%
(\lambda,\sigma_0,\sigma_1,\ldots,\sigma_{r+n})$, the polynomial $x\mapsto
J_k(x):= \sum_{\beta \in \mathbb{N}_{2k}^n}\lambda_{\beta}x^{\beta}$ is of
degree $2k$ and it satisfies, for all  $x \in \Omega=\{x:\theta_l(x) \le 0, l=1,\ldots,n\}$ and $y \in F:=\{w: h_j(w) \le 0, j=1,\ldots,r\}$,
\begin{equation*}
G(x,y)-\sum_{\beta \in \mathbb{N}_{2k}^n} \lambda_{\beta}  {x}%
^{\beta}=\sigma_0( {x}, {y})-\sum_{j=1}^{r} \sigma_j( {x},%
 {y}) h_j( {y})-\sum_{l=1}^n \sigma_{r+l}(x,y) \theta_l(x) \ge 0.
\end{equation*}
So, for every $k \in \mathbb{N}$, $J_k(x) \le J(x):=\min_{w \in \mathbb{R}^m}\{G(x,w): h_j(w) \le
0\}$ for all  $x \in \Omega$. Indeed, the next theorem shows that $J_k$ converges to the optimal
value function $J$ on $\Omega$, in the $L_1$-norm sense.

\begin{lemma}\label{lemma:4.4}
\textrm{(\cite{Lasserre2010})} Suppose that Assumption 2.1 holds. For each $%
k\in \mathbb{N}$, let $\rho _{k}$ be the optimal value of the semidefinite
programming (\ref{eq:2.2}). Let $\epsilon _{k}\rightarrow 0$ and let $%
(\lambda ,\sigma _{0},\sigma _{1},\ldots ,\sigma _{r+n})$ be an $\epsilon _{k}$%
-solution of (\ref{eq:2.2}) in the sense that $\sum_{\beta \in \mathbb{N}%
_{2k}^{n}}\lambda _{\beta }\gamma _{\beta }\geq \rho _{k}-\epsilon _{k}$.
Define $J_{k}\in \mathbb{R}_{2k}[x]$ by
$J_{k}(x)=\sum_{\beta \in \mathbb{N}_{2k}^{n}}\lambda _{\beta }x^{\beta}$.
%\end{equation*}%
Then, we have $J_{k}(x)\leq J(x)$ for all $x \in \Omega$ and $$\int_{\Omega }|J_{k}(x)-J(x)|d\varphi (x)\rightarrow 0  \mbox{ as } k\rightarrow \infty .$$
%\begin{equation*}
%
%\end{equation*}
\end{lemma}

We now introduce a scheme to solve the $\epsilon$-approximation problem for
arbitrary $\epsilon>0$, using sequences of semidefinite programming
relaxations.

\textbf{Algorithm 4.5 (general scheme)}

Step 0: Fix $\epsilon> 0$. Set $k=1$.

Step 1: Solve the semidefinite programming problem (\ref{eq:2.2}) and obtain
the $\frac{1}{k}$-solution $(\lambda^k,\sigma_j^k)$ of (\ref{eq:2.2}).
%and its dual
%\begin{eqnarray*}
%(D_{1k}) &\displaystyle \sup_{p,\sigma_0,\ldots,\sigma_j}& \sum_{\alpha}p_{\alpha}\, {\bf x}^{\alpha} \\
%& \mbox{subject to } & G-p=\sigma_0+\sum_{j=1}^r \sigma_j h_j \\
%& & \sigma_j \in \Sigma[x,w], j=0,1,\ldots,r, \ p \in \mathbb{R}[x] \\
%& & {\rm deg} p \le 2k, {\rm deg} (\sigma_j h_j) \le 2k
%%\langle \sum_{\alpha}p_{\alpha}, {\bf x}^{\alpha}\rangle
%\end{eqnarray*}
Define $J_{k}(x)= \sum_{\beta \in \mathbb{N}_{2k}^n}\lambda_{\beta}^k
x^{\beta}$.

Step 2: Consider the following semialgebraic set
\begin{equation*}
S_k:= \{(x,y): g_i(x,y) \le 0, \ i=1,\ldots,s, h_j(y) \le 0, \ j=1,\ldots,r,
G(x,y)-J_{k}(x) \le \epsilon\}.
\end{equation*}
If $S_k=\emptyset$, then let $k=k+1$ and return to Step 1. Otherwise, go to
Step 3.

Step 3: Solve the following polynomial optimization problem
\begin{eqnarray*}
(P_{\epsilon}^k) &\displaystyle \min_{(x,y) \in \mathbb{R}^n \times \mathbb{R%
}^m}& f(x,y) \\
& \mbox{subject to } & g_i(x,y) \le 0, \ i=1,\ldots,s, \\
& & h_j(y) \le 0, \ j=1,\ldots,r, \\
& & G(x,y)-J_{k}(x) \le \epsilon.
\end{eqnarray*}

Step 4: Let $v_{\epsilon}^k=\min_{1 \le i \le k} \mathrm{val}(P_{\epsilon}^i)
$. Update $k=k+1$. Go back to Step 1.

\bigskip

Before we establish the convergence of this procedure, let us comment that the feasibility problem of the
semialgebraic set in Step 2 can be tested by a sequence of SDP relaxations via the
Positivstellensatz. This was explained in \cite{Palbo} and was implemented
in the matlab toolbox SOSTOOLS. As explained before, Step 3 can also be
accomplished by solving a sequence of SDP relaxations.

Let us show that there exists a finite number $k_0$ such that $%
S_{k_0} \neq \emptyset$, and so,  Algorithm 4.5 is well-defined.

\begin{lemma}\label{lemma:6}
Let $\epsilon>0$. Consider the problem $(P_{\epsilon})$ and Algorithm 4.5. Let $%
K=\{(x,y):g_i(x,y) \le 0, i=1,\ldots,s\}$ and $F=\{w: h_j(w) \le 0, j=1,\ldots,r\}$. Suppose
that Assumption 2.1 holds and $\mathrm{cl}\big(\mathrm{int}(K \cap (\mathbb{R}^n \times F))\big)=K \cap (\mathbb{R}^n \times F)$. Then, there exists a finite
number $k_0$ such that $S_{k_0} \neq \emptyset$ in Step 2 of Algorithm 4.5.
\end{lemma}

\begin{proof}
Note from Corollary \ref{cor:1} that a global minimizer $(\bar x,\bar y)$ of $(P)$ exists.
In particular, the set
$D_0:=\{(x,y) \in  K \cap (\mathbb{R}^n \times F): G(x,y)-J(x)< \epsilon\}$
is an nonempty set as $(\bar x, \bar y) \in D_0$.  Noting from our assumption, we have
$\mathrm{cl}\big(\mathrm{int}(K \cap (\mathbb{R}^n \times F))\big)=K \cap (\mathbb{R}^n \times F)$.   This together with the fact that
$\{(x,y): G(x,y)-J(x)< \epsilon\}$ is an open set (as the optimal value function of the lower level problem
$J(x)$
is continuous) gives us that
\[
\tilde{D}:=\{(x,y) \in \mathrm{int}(K \cap (\mathbb{R}^n \times F)): G(x,y)-J(x)< \epsilon\}
\]
is a nonempty open set. Define  $D:={\rm Pr}_1 \tilde{D}=\{x \in \mathbb{R}^n: (x,y) \in \tilde{D} \mbox{ for some } y \in \mathbb{R}^m\}$. Then, $D$ is also
a nonempty open set.
%In particular, for any $\delta \in (0,\epsilon)$, $$S=\{(x,y): g_i(x,y) \le 0, \ i=1,\ldots,s,  h_j(y) \le 0, \ j=1,\ldots,r, G(x,y)-J(x) \le \delta\}\neq \emptyset,$$
%where $J(x):=\min\{G(x,y): h_j(y) \le 0\}$.
Note from Lemma \ref{lemma:4.4} that $J_k$ converges to $J$ in $L^1(\Omega,\varphi)$-norm. Hence $J_k$ converges to $J$ almost everywhere on $\Omega$. As $\varphi (\Omega)<+\infty$,  the classical Egorov's theorem\footnote{The Egorov's theorem \cite[Theorem 2.2]{Analysis} states that: for a measure space $(\Omega,\varphi)$, let $f_k$ be a sequence of functions on $\Omega$. Suppose that $\Omega$ is of finite $\varphi$-measure and $\{f_k\}$ converges $\varphi$-almost everywhere on $\Omega$ to a limit function $f$. Then, there exists
a subsequence $l_k$ such that $f_{l_k}$ converges to $f$ almost uniformly in the sense that, for every $\epsilon > 0$, there exists a measurable subset $A$ of $\Omega$ such that $\varphi(A) < \epsilon$, and $\{f_{l_k}\}$ converges to $f$ uniformly on the relative complement $\Omega \backslash A$.} implies that there exists a subsequence
$l_k$ such that $J_{l_k}$ converges to $J$ $\varphi$-almost uniformly on $\Omega$.  So, there exists
 a Borel set $A$ with $\varphi(A)<\frac{\eta}{2}$ with $\eta:=\varphi(D) >0$ such that
\[
J_{l_k} \rightarrow J \mbox{ uniformly over }  \Omega \backslash A.
\]
We observe that
$(\Omega \backslash A)\cap D \neq \emptyset$ (Otherwise, as $D \subseteq {\rm Pr}_1 K \subseteq \Omega$, we have $D \subseteq A$. This implies that $\eta=\varphi(D) \le \varphi(A)=\eta/2$ which is impossible as $\eta>0$).
Let $x_0 \in (\Omega \backslash A)\cap D$. Then, we have $J_{l_k}(x_0) \rightarrow J(x_0)$ and there exists $y_0 \in \mathbb{R}^m$ such that
$y_0 \in F$, $G(x_0,y_0)-J(x_0)< \epsilon$. In particular, for all $k$ large, $(x_0,y_0) \in S_{l_k}$. Therefore,
$S_{l_k} \neq \emptyset$ for all large $k$, and so, the conclusion follows.
\end{proof}

\begin{remark}
The fact that  $L_1$-convergence implies the almost-uniform convergence can also be seen by using Theorem 2.5.1
($L_1$-convergence implies convergence in measure) and Theorem 2.5.3
(convergence in measure implies almost-uniform convergence for a subsequence) of \cite[Page 92-93]{Ash} without requiring
the measure of $\Omega$ to be finite.
\end{remark}

We note that the condition ``$\mathrm{cl}\big(\mathrm{int}(K \cap (\mathbb{R}^n \times F))\big)=K \cap (\mathbb{R}^n \times F)$'' holds
when $C:=K \cap (\mathbb{R}^n \times F)$ is a finite union of
closed convex sets $C_i$ with $\mathrm{int}C_i \neq \emptyset$.
Moreover, if the set $C$ is of the form $\{(x,y) \in \mathbb{R}^{n} \times \mathbb{R}^m: G_i(x,y) \le 0, i=1,\ldots,l\}$ for some
polynomials $G_i$, $i=1,\ldots,l$ and $l \in \mathbb{N}$, then the above condition also holds if the commonly used Mangasarian-Fromovitz constraint
qualification \cite{MFCQ} is satisfied for any $(x,y) \in C$.

We are now ready to state the convergence theorem of the proposed Algorithm 4.5.  The proof of it is quite technical and so it is given later in Appendix C.
%and has been delay to the Appendix C.

\begin{theorem}\label{th:convergence}
\textbf{(General bilevel problem (P): Convergence theorem)} Let $\epsilon >0$ and consider
problem $(P_{\epsilon })$. Let $v_{\epsilon }^{k}$ be generated by Algorithm 4.5.  Let $%
K=\{(x,y):g_i(x,y) \le 0, i=1,\ldots,s\}$ and $F=\{w: h_j(w) \le 0, j=1,\ldots,r\}$. Suppose
that Assumption 2.1 holds and  $\mathrm{cl}\big(\mathrm{int}(K \cap (\mathbb{R}^n \times F))\big)=K \cap (\mathbb{R}^n \times F)$.  Then,
\begin{itemize}
\item[\textrm{(i)}] $v_{\epsilon}^k \rightarrow v_{\epsilon}$ as $k
\rightarrow \infty$ where $\mathrm{val}(P_{\epsilon}) \le v_{\epsilon} \le %
\displaystyle \lim_{\delta \rightarrow \epsilon^{-}}\mathrm{val}(P_{\delta})$%
. In particular, for almost every $\epsilon$, $v_{\epsilon}^k \rightarrow \mathrm{%
val}(P_{\epsilon})$ in the sense that, for all finite intervals $I \subseteq
\mathbb{R}_+$, $v_{\epsilon}=\mathrm{val}(P_{\epsilon})$ for all $\epsilon
\in I$ except at finitely many points.

\item[\textrm{(ii)}] There exists $\epsilon_0>0$ such that, for all $\epsilon \in (0,\epsilon_0)$, $v_{\epsilon}^k
\rightarrow \mathrm{val}(P_{\epsilon})$ as $k \rightarrow \infty$%
. Moreover, let $\delta_k \downarrow 0$. Let $v_{\epsilon}^k=\min_{1 \le i
\le k} \mathrm{val}(P_{\epsilon}^i)=\mathrm{val}(P_{\epsilon}^{i_k})$ and
let $(x_k,y_k)$ be a $\delta_k$-solution of $(P_{\epsilon}^{i_k})$. Then, $%
\{(x_k,y_k)\}$ is a bounded sequence and any cluster point $(\widehat{x},%
\widehat{y})$ of $(x_k,y_k)$ is a global minimizer of $(P_{\epsilon})$ for
all $\epsilon \in (0,\epsilon_0)$.
% \item[{\rm (3)}] For each fixed $\epsilon>0$, there exists $l_0$ $($depends on $\epsilon)$ such that for all $l \ge l_0$
% $${\rm val}(P_{2,k_l})< {\rm val}(P)+\epsilon.$$
%Suppose that either $\Phi$ is lower semicontinuous or
%$m(S_{k_l})>0$ for all large $l$ where $S_{k_l}$ is the solution set of the sequential approximation problem $(P_{2,k_l})$. Then,  there exists $l_0$ such that for all $l \ge l_0$
% $${\rm val}(P)-\epsilon < f_{k_l}^{\epsilon} < {\rm val}(P)+\epsilon.$$
\end{itemize}
\end{theorem}

We now  illustrate how our general scheme can lead to solving a bilevel
programming problem with a nonconvex lower level problem towards a global
solution. This is done by applying our scheme to a known test problem of the bilevel programming literature.
\begin{example}\label{example:4.8}{\bf (Illustration of our approximation scheme)}
Consider the following bilevel optimization test problem (for example see \cite{Ye,Barton})
\begin{eqnarray*}
& \min_{(x,y) \in \mathbb{R}^2 } & x+y \\
& \mbox{subject to} & x \in [-1,1], y \in {\rm argmin}_{w \in [-1,1]}\{\frac{xw^2}{2}-\frac{w^3}{3}\}.
\end{eqnarray*}
Let $Y(x):={\rm argmin}_{w \in [-1,1]}\{\frac{xw^2}{2}-\frac{w^3}{3}\}$.  Clearly, the lower level problem is nonconvex and all the conditions in Theorem \ref{th:convergence} are
satisfied. The optimal value function of the lower level problem is given by
\[
J(x)= \min_{w \in [-1,1]}\{\frac{xw^2}{2}-\frac{w^3}{3}\}=\left\{\begin{array}{ccl}
                                                                 0, & \mbox{ if } & x \in [\frac{2}{3},1], \\
                                                                 \frac{x}{2}-\frac{1}{3}, & \mbox{ if } & x \in [-1,\frac{2}{3}),
                                                                 \end{array}
 \right.
\]
and the solution set of the lower level problem $Y(x)$ can be formulated as
\[
Y(x)= \left\{\begin{array}{ccl}
                                                                 \{0\}, & \mbox{ if } & x \in (\frac{2}{3},1], \\
                                                                 \{0,1\}, & \mbox{ if } & x=\frac{2}{3}, \\
                                                                 \{1\}, & \mbox{ if } & x \in [-1,\frac{2}{3}).
                                                                 \end{array}
 \right.
 \]
It is easy to check that the true (unique) global minimizer is $(\bar x,\bar y)=(-1,1)^T$ and the true global optimal value is $0$.

Now, for $k=3$, using GloptiPoly 3, we obtain a degree $2k(=6)$ polynomial approximation of $J(x)$ which is
\[
J_3(x)= -0.3338+ 0.5011*x+0.0098*x^2-0.0032*x^3-0.0696*x^4-0.1012*x^5-0.0432*x^6.
\]
The following figure depicts the graph of the functions $J_3$ and $J$, where the red curve is the graph of the function $J$ and the blue curve is the graph of the degree
$6$ polynomial
$J_3$. From the graph, we can see that $J_3 \le J$ over the interval $[-1,1]$ and provides
a reasonably good approximation of the piecewise differentiable (and so, non-polynomial) function $J(x)$.
\begin{center}
{\bf Figure 1: $J(x)$ and its degree-6 underestimation in Example 4.8}

\includegraphics[scale=0.50]{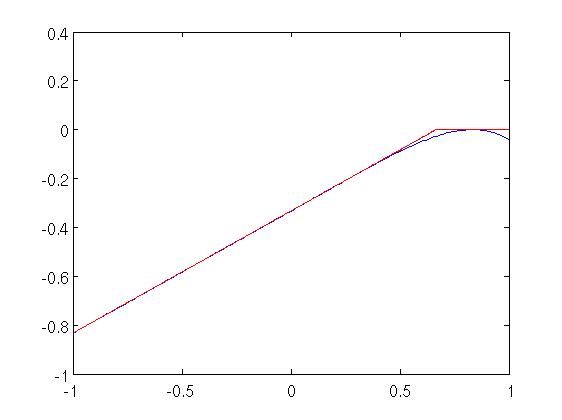}
	\end{center}

Setting $\epsilon=0.001$ and solving the following polynomial optimization problem
\begin{eqnarray*}
& \min_{(x,y) \in \mathbb{R}^2 } & x+y \\
& \mbox{subject to} & x \in [-1,1], \\
& & y \in [-1,1], \\
& & \frac{xy^2}{2}-\frac{y^3}{3}-J_3(x) \le 0.001,
\end{eqnarray*}
with GloptiPoly 3,
% \begin{verbatim}
%sdpvar x y
%obj1=x+y
%F=[1-x^2>=0;1-y^2>=0;0.5*x*y^2-(1/3)*y^3-...
%(-0.3338+ 0.5011*x+0.0098*x^2-0.0032*x^3-0.0696*x^4-0.1012*x^5-0.0432*x^6)<=0.001]
%[sol,x] = solvemoment(F,obj1,[],4);
%\end{verbatim}
the solver returns the point $(x,y)=(-1.0000,0.9996)$ with its associated function value $-4.1680e-04$, which is a reasonably good
approximation of the true global minimizer and global optimal value of the bilevel programming problem.

% The following table provides the value of $v_{\epsilon}^k$ and the associated solutions in the first 4 iterations with $\epsilon=0.001$ for Algorithm 4.5.
% \begin{center}
% \begin{tabular}{||c|c|c||} \hline
%  \mbox{iteration} $k$ & value of $v_{\epsilon}^k$ &  the associated solution \\ \hline
% 1 & 0.7008 & (-0.2992,1.000) \\ \hline
% 2 & 0.7008 & (-0.2992,1.000) \\ \hline
% 3 & -4.1680e-04 & (-1.0000,0.9996) \\ \hline
% 4 & -2.1680e-04 & (-1.0000,0.9998) \\ \hline
% \end{tabular}
% \end{center}
\end{example}

\begin{remark}{\bf (Further extensions of the approach)} \label{remark:3} Although we presented our approach for a class of bilevel problems where the
constraints of the lower-level problem are independent of the upper-level
decision variable $x$, our approach may be extended to solve the following more
general bilevel polynomial optimization problem:
\begin{eqnarray*}
(GP) &\displaystyle\min_{x\in \mathbb{R}^{n},y\in \mathbb{R}^{m}}&f(x,y) \\
&\mbox{subject to }&g_i(x,y)\leq 0,\ i=1,\ldots ,s, \\
&&y\in Y(x):=\displaystyle\mathrm{argmin}_{w\in \mathbb{R}%
^{m}}\{G(x,w):h_{j}(x,w)\leq 0,j=1,\ldots ,r\},
\end{eqnarray*}%
where the constraints of the lower level problem are allowed to depend on $x$. In this case, we can construct a sequence of semidefinite programming relaxation for
finding a global minimizer and a global minimum of its $\epsilon $%
-approximation problem and  similar convergence results of the scheme can be
achieved under an additional technical assumption that the optimal value
function of the lower level problem $J(x):=\min_{w\in \mathbb{R}%
^{m}}\{G(x,w):h_{j}(x,w)\leq 0,j=1,\ldots ,r\}$ is continuous. However, we
wish to note that, for the problem (P) discussed in this paper (that
is, $h_{j}$ are independent of $x$), this condition is automatically
satisfied. On the other hand, in general, this condition may fail for the
general problem (GP) even when $n=m=1$. We provide a simple example to
illustrate this.
Consider the following bilevel programming problem
\begin{eqnarray*}
 &\displaystyle \min_{x \in \mathbb{R}, y \in \mathbb{R}}& x^2+y^2 \\
& \mbox{subject to } & 0 \le x \le 1, \\
& & y \in Y(x):=\displaystyle {\rm argmin}_{w \in \mathbb{R}}\{(x-w)^2: x^2-w^2 \le 0, w(w-1) \le 0, -w(w-1) \le 0\}.
\end{eqnarray*}
It can be directly verified that the optimal value function of the lower level problem is given by $$J(x):=\min_{w \in \mathbb{R}}\{(x-w)^2: x^2-w^2 \le 0, w(w-1) = 0\}=\left\{\begin{array}{ccc}
 0,    & \mbox{ if } & x =0, \\
(x-1)^2,   &  \mbox{ if } & x \in (0,1].                                                                                                                                                                 \end{array}
 \right.$$
and is discontinuous at $x=0$.
\end{remark}

\section{Numerical Examples}

In this Section, we apply our schemes to solve some
bilevel optimization test problems available in the literature and present their results. We conducted the numerical tests on a computer with a 2.8 GHz Intel Core i7  and 8 GB RAM, equipped with Matlab 7.14 (R2012a). We solved bilevel polynomial problems with convex as well as non-convex lower-level problems, where the lower level problems
are independent of the upper level decision variables.

We first present results for the following bilevel problems with a convex lower level problem. We note that all the assumptions of Theorem \ref{th:90} are satisfied by these bilevel problems with a convex lower level problem.

\begin{example}\label{ex:5.1}
Consider the following bilevel polynomial problem \cite{test}
\[
\left\{\begin{array}{ll}
\min_{x,y \in \mathbb{R}} & (x-3)^2+(y-2)^2 \\
\mbox{s.t.} & -2x+y-1 \le 0 \\
& x-2y+2 \le 0 \\
& x+2y-14 \le 0 \\
& 0 \le x \le 8 \\
& 0 \le y \le 10 \\
&  y \in \displaystyle\mathrm{argmin}_{w\in \mathbb{R}%
}\{(w-5)^2: w \in [0,10]\}.
                     \end{array}
 \right.
\]
This problem has a unique global minimizer
$(x^*,y^*)=(3,5)$ and the optimal value $f^*=9$.
\end{example}

\begin{example}\label{ex:5.2}
Consider the following bilevel polynomial problem \cite{test}
\[
\left\{\begin{array}{ll}
\min_{x,y \in \mathbb{R}} & -(4x-3)y+(2x+1)\\
\mbox{s.t.} &  0 \le x \le 1 \\
& 0 \le y \le 1 \\
&  y \in \displaystyle\mathrm{argmin}_{w\in \mathbb{R}%
}\{ -(1-4x)w-(2x+2): w \in [0,1]\}.
                     \end{array}
 \right.
\]
This problem has a unique global minimizer
$(x^*,y^*)= (0.25,0)$ and the optimal value $f^*=1.5$.
\end{example}
We first transformed the problems in Example \ref{ex:5.1} and Example \ref{ex:5.2}
into equivalent single-level nonconvex polynomial optimization problems as proposed in Section 3. Then, we used GloptiPoly 3 \cite{Gloptpoly} and the SDP solver Sedumi \cite{Sedumi}
to solve the transformed polynomial optimization problems.
For these two problems, the second relaxation problem (that is, problem $(Q_2)$) of the SDP approximation scheme (3.12) returns a solution which agrees with the true solution.

The following table summarizes the results of bilevel problems with a convex lower level problem where $(x^*,y^*)$ and $f^*$
denote the true global minimizer and the true optimal value respectively,  $(x,y)$ and $f$ denote the computed minimizer and the computed optimal value respectively and
CPU time represents the CPU time (in seconds) used to solve the problems.
\begin{center}
{{\bf{Table 1: Convex Lower-Level Problems}}}

{\small \begin{tabular}{||c|c|c||} \hline
Test Problems   & Known optimal solutions & Computed solutions \\ \hline
Example \ref{ex:5.1}  &   $(x^*,y^*)=(3,5)$   & $(x,y)=(3.0000,5.0000)$  \\
& $f^*=9$ &  $f=9.0000$ \\
& & CPU time=0.2511\\
\hline
  Example \ref{ex:5.2}   & $(x^*,y^*)= (0.25,0)$   & $(x,y)= (0.2500,0.0000)$  \\
& $f^*=1.5$ &  $f=1.5000$  \\
& & CPU time=0.1957
\\ \hline
%  Example \ref{ex:5.3}   &   $(x^*,y^*)= (0.2106,1.799)$   & $(x,y)= $  \\
%  & $f^*=-1.755$ &  $f^*=$ \\ \hline
%   Example \ref{ex:5.4}   &   $(x^*,y^*)= (-0.25,\pm 0.5 )$  & $(x,y)= $  \\
%  & $f^*=0.1875$ &  $f^*=$ \\ \hline
\end{tabular} }

\end{center}

We now solve the following bilevel problems with a non-convex lower level problem. Again, all the assumptions in Theorem \ref{th:convergence} are satisfied by these bilevel problems with a nonconvex lower level problem.
\begin{example}\label{ex:5.4}
Consider the following bilevel polynomial problem \cite{test1}
\[
\left\{\begin{array}{ll}
\min_{x,y \in \mathbb{R}} & x \\
\mbox{s.t.} &  -x+y \le 0 \\
& -10 \le x \le 10 \\
& -1 \le y \le 1 \\
&  y \in \displaystyle\mathrm{argmin}_{w\in \mathbb{R}%
}\{w^3: w \in [-1,1]\}.
                     \end{array}
 \right.
\]
This problem has a unique global minimizer
$(x^*,y^*)=(-1,-1)$ with the optimal value $f^*=-1$.
\end{example}

\begin{example}\label{ex:5.5}
Consider the following bilevel polynomial problem \cite{test1}
\[
\left\{\begin{array}{ll}
\min_{x,y \in \mathbb{R}} & 2x+y \\
\mbox{s.t.} &  -1 \le x \le 1 \\
& -1 \le y \le 1 \\
&  y \in \displaystyle\mathrm{argmin}_{w\in \mathbb{R}%
}\{-\frac{1}{2}xw^2-\frac{1}{4}w^4: w \in [-1,1]\}.
                     \end{array}
 \right.
\]
This problem has  two global minimizers
$(x_1^*,y_1^*)= (-1,0)$ and $(x_2^*,y_2^*)= (-1/2,-1)$ with the optimal value $f^*=-2$.
\end{example}
%\begin{verbatim}
%sdpvar x0 x1 x2 x3 x4 x5 x6 x7 x8 x9 x10 x11 x12
%sdpvar x y
%obj=x0+(1/3)*x2+(1/5)*x4+(1/7)*x6 +(1/9)*x8+(1/11)*x10
%[s1,c1]=polynomial([x,y],10)
%F=[sos(-0.5*x*y^2-0.25*y^4-(x0+x1*x+x2*x^2+x3*x^3+x4*x^4+x5*x^5+x6*x^6+x7*x^7+x8*x^8+x9*x^9+x10*x^10)+s1*(y^2-1)),sos(s1)];
%%+x7*x^7+x8*x^8+x9*x^9+x10*x^10+x11*x^11+x12*x^12
%
%solvesos(F,-obj,[],[c1;x0;x1;x2;x3;x4;x5;x6;x7;x8;x9;x10])%;x7;x8;x9;x10;x11;x12])
%
%x0=double(x0)
%x1=double(x1)
%x2=double(x2)
%x3=double(x3)
%x4=double(x4)
%x5=double(x5)
%x6=double(x6)
% x7=double(x7)
% x8=double(x8)
% x9=double(x9)
% x10=double(x10)
%x11=double(x11)
%x12=double(x12)
%
%mpol x y
%f=2*x+y
%K=[x>=-1; x<=1; y>=-1; y<=1;
%    -0.5*x*y^2-0.25*y^4-(x0+x1*x+x2*x^2+x3*x^3+x4*x^4+x5*x^5+x6*x^6+x7*x^7+x8*x^8+x9*x^9+x10*x^10)<=0.004]%+x7*x^7+x8*x^8+x9*x^9+x10*x^10+x11*x^11+x12*x^12)>=0]
%P=msdp(min(f),K,8)
%[status,obj]=msol(P)
%\end{verbatim}

\begin{example}\label{ex:5.6}
Consider the following bilevel polynomial problem \cite{test1}
\[
\left\{\begin{array}{ll}
\min_{x,y \in \mathbb{R}} & y \\
\mbox{s.t.} &  0.1 \le x \le 1 \\
& -1 \le y \le 1 \\
&  y \in \displaystyle\mathrm{argmin}_{w\in \mathbb{R}%
}\{x(16 w^4 + 2 w^3 + 8 w^2 +\frac{3}{2} w + \frac{1}{2}): w \in [-1,1]\}.
                     \end{array}
 \right.
\]
This problem has infinitely many global minimizers
$(x^*,y^*)= (a,0.5)$ for any $a \in [0.1,1]$ with the optimal value $f^*=0.5$.
%
%$\min_{x,y \in \mathbb{R}} y
%\mbox{s.t} y \in{\rm argmin}_{w}
%x ¸ [0.1, 1], y, z ¸ [.1, 1]$
%has infinitely many optimal solution points x ¸ [0.1, 1], y = 0.5 with an optimal objective value of 0.5.
\end{example}
%\begin{verbatim}
%solvesos(F,-obj,[],[c1;x0;x1;x2;x3;x4;x5;x6;x7;x8;x9;x10])%;x7;x8;x9;x10;x11;x12])
%
%x0=double(x0)
%x1=double(x1)
%x2=double(x2)
%x3=double(x3)
%x4=double(x4)
%x5=double(x5)
%x6=double(x6)
% x7=double(x7)
% x8=double(x8)
% x9=double(x9)
% x10=double(x10)
%x11=double(x11)
%x12=double(x12)
%
%mpol x y
%f=y
%K=[x>=0.1; x<=1; y>=-1; y<=1;
%    x*(16*y^4 + 2*y^3 - 8*y^2 - 1.5*y + 0.5)-(x0+x1*x+x2*x^2+x3*x^3+x4*x^4+x5*x^5+x6*x^6+x7*x^7+x8*x^8+x9*x^9+x10*x^10)<=0.000001]%+x7*x^7+x8*x^8+x9*x^9+x10*x^10+x11*x^11+x12*x^12)>=0]
%P=msdp(min(f),K,8)
%[status,obj]=msol(P)
%\end{verbatim}
\begin{example}\label{ex:5.7}
Consider the following bilevel polynomial problem \cite{test1}
\[
\left\{\begin{array}{ll}
\min_{x,y \in \mathbb{R}} & -x + xy + 10y^2 \\
\mbox{s.t.} &  -1 \le x \le 1 \\
& -1 \le y \le 1 \\
&  y \in \displaystyle\mathrm{argmin}_{w\in \mathbb{R}%
}\{-xw^2+w^4/2: w \in [-1,1]\}.
                     \end{array}
 \right.
\]
This problem has  a unique global minimizer
$(x^*,y^*)= (0,0)$  with the optimal value $f^*=0$.
\end{example}

 We solved these four problems by using the approximation scheme proposed
 in Section 4 implemented via the software GloptiPoly 3 and the SDP solver Sedumi. For detailed illustration of how the scheme is implemented, see Example
\ref{example:4.8}.  %The numerical results are summarized in the following table where
The numerical results are summarized in the following table. Note that $\textit{{\rm deg}}$ denotes the maximum degree of the polynomial underestimation used in a subproblem of our scheme.
%These problems are solved using the approximation scheme proposed in Section 4.
%implemented via the software GloptiPoly 3 and the SDP solver Sedumi.
%
 %In the computation, we terminate the scheme if the error between the computed optimal value of the problem and the known optimal value
% is less or equal to $5 \times 10^{-3}$. For all of these four examples, the scheme terminates at most in the 7th iteration and produces a reasonably good approximation of a true global solution.

\begin{center}
{\bf{Table 2: Non-Convex Lower-Level Problems}}

{\small \begin{tabular}{||c|c|c||} \hline
Test Problems  & Known optimal solutions & Computed solutions \\ \hline
% Example \ref{ex:5.1}  & Convex &   $(x^*,y^*)=(3,5)$   & $(x,y)=(3.0000,5.0000)$  \\
% & & $f^*=9$ &  $f=9.0000$ \\
% & &  &  Relaxation order: 2  \\ \hline
%   Example \ref{ex:5.2}   &  Convex & $(x^*,y^*)= (2.5,0)$   & $(x,y)= (2.5000,0.0000)$  \\
% & & $f^*=1.5$ &  $f=1.5000$ \\
% & &  &  Relaxation order: 2  \\ \hline
%  Example \ref{ex:5.3}   &   $(x^*,y^*)= (0.2106,1.799)$   & $(x,y)= $  \\
%  & $f^*=-1.755$ &  $f^*=$ \\ \hline
%   Example \ref{ex:5.4}   &   $(x^*,y^*)= (-0.25,\pm 0.5 )$  & $(x,y)= $  \\
%  & $f^*=0.1875$ &  $f^*=$ \\ \hline
%   Example \ref{ex:5.3}   &  Non-convex & $y^*= 0.5$ & $y= 0.4998$  \\
%& & $f^*=0.5$ &  $f=0.4998$ \\
%& &  &  Iteration: 6  \\ \hline
   Example \ref{ex:5.4}   & $(x^*,y^*)= (-1,-1)$ & $(x,y)= (-1.0000,-1.0000)$  \\
& $f^*=-1$ &  $f=-1.0000$ \\
& & {\rm CPU time}=1.0746 \\
& &   {\rm deg}=12  \\
\hline
   Example \ref{ex:5.5}  & $(x^*,y^*)= (-1,0)$ or $(-1/2,-1)$  & $(x,y)= (-0.9991,-0.0020)$  \\
& $f^*=-2$ &  $f=-2.0002$ \\
& & {\rm CPU time}=5.1432 \\
& &   {\rm deg}=14  \\
\hline
   Example \ref{ex:5.6}   &  $(x^*,y^*)= (a,0.5)$ for all $a \in [0.1,1]$ & $(x,y)= (0.2299,0.4990)$ \\
& $f^*=0.5$ &  $f=0.4990$ \\
& & {\rm CPU time}=6.8819 \\
& &   {\rm deg}=12  \\
\hline
    Example \ref{ex:5.7}   & $(x^*,y^*)= (0,0)$ & $(x,y)= (0.0034,-0.0002)$ \\
& $f^*=0$ &  $f=-0.0034$  \\
& & {\rm CPU time}= 0.8844 \\
& &   {\rm deg}=10  \\
\hline
\end{tabular} }
\end{center}
\section{Conclusion and Further Research}

We established that a global minimizer and the global minimum of a
bilevel polynomial optimization problem can be found by way of solving a sequence of
semidefinite programming relaxations.  We first considered a bilevel polynomial
optimization problem where the lower level problem is a convex problem. In this case, we
proved that the values of the sequence of relaxation problems converge to the global optimal value of the bilevel problem under a mild assumption. This shows that a global solution can simply be found by first
transforming the bilevel problem  into an equivalent single-level polynomial
problem and then solving the resulting single-level problem by the standard
sequential SDP relaxations used in the polynomial optimization \cite{Lasserre_book}.

We then examined a general bilevel polynomial optimization problem with a
not necessarily convex lower-level problem. We established that the global optimal value in this case can be found by way of solving a new sequential
semidefinite programming relaxation problems  based on the joint-marginal approach proposed in \cite%
{Lasserre2010}. This was done by using a sequence of semidefinite
programming relaxations of its $\epsilon$-approximation problem  under the standard
Assumption 2.1 of polynomial optimization, where  $%
\epsilon >0$ is smaller than a positive threshold.

The convergence of the proposed semidefinite programming
approximation scheme relies on Assumption 2.1 which requires that the
feasible set of the bilevel problem is bounded. The proposed scheme can also
be extended to cover possible unbounded feasible sets by exploiting
coercivity of the objective function of the upper/lower level problem as
in our recent papers \cite{JKLL,JLL_JOTA,JPL} where the convergence of
the sequence of semidefinite programming relaxations was established for polynomial
optimization problems with unbounded feasible sets.

Our bilevel problem, in the present paper,
represents the so-called optimistic approach to the leader and follower's
game in which the follower is assumed to be co-operative and so, the leader
can choose the solution with the lowest cost. The pessimistic approach assumes that the follower may not be co-operative
and hence the leader will need to prepare for the worst cost. Mathematically, the following bilevel problem represents the pessimistic approach:
\begin{eqnarray*}
 &\displaystyle\min_{x\in \mathbb{R}^{n}}& \max_{y \in Y(x)} f(x,y) \\
&\mbox{subject to }& g_i(x)\leq 0,\ i=1,\ldots ,s,
\end{eqnarray*}
where $Y(x):=\displaystyle\mathrm{argmin}_{w\in \mathbb{R}%
^{m}}\{G(x,w):h_{j}(w)\leq 0,j=1,\ldots ,r\}$. A possible method to solving this bilevel problem is to construct a polynomial approximation for the optimal value of the problem $x \mapsto \max_{y \in Y(x)} f(x,y)$ using the joint marginal approach of \cite{Lasserre2010} and then design a semidefinite programming approximation method that is similar to the scheme studied in the present paper. This would be an interesting topic for future research.

\vspace{-0.3cm}
\section*{Appendix A: Semi-algebraic functions and sets}
In this appendix, we summarize some of the important properties of semi-algebraic functions which are used in this paper (see \cite{Bierstone1988}).
\begin{itemize}
\item[{\rm (i)}] Finite union (resp. intersection) of semi-algebraic
sets is semi-algebraic. The Cartesian product (resp. complement, closure) of semi-algebraic
sets is semi-algebraic.

\item[{\rm (ii)}] If
$f, g$
are semi-algebraic functions on $\mathbb{R}^n$
and $\lambda \in \mathbb{R}$, then $f+g$, $f g$ and $\lambda f$ are all semi-algebraic functions.

\item[{\rm (iii)}] If
$f$
is a semi-algebraic function on $\mathbb{R}^n$
and $\lambda \in \mathbb{R}$, then $\{x: f(x) \le \lambda\}$ (resp. $\{x: f(x) \le \lambda\}$, $\{x: f(x) < \lambda\}$ and $\{x: f(x) = \lambda\}$ are all
semi-algebraic sets.

\item[{\rm (iv)}] A composition of semi-algebraic maps is a semi-algebraic map.

\item[{\rm (v)}]  The  image  and  inverse image of  a  semi-algebraic set  under  a  semi-algebraic map are semi-algebraic sets. In particular, the projection of a semi-algebraic set
is still a semi-algebraic set.

\item[{\rm (vi)}]  If $S$ is a compact semi-algebraic set in $\mathbb{R}^m$ and $f: \mathbb{R}^n \times \mathbb{R}^m \rightarrow \mathbb{R}$ is a real polynomial,
then the function $x \mapsto \displaystyle \min_{y \in \mathbb{R}^m} \{f(x,y)\ : \ y \in S \},$ is also semi-algebraic.

\end{itemize}
\begin{remark}\label{remark:1}{\rm If  $A \in \mathbb{R}^n ,  B \in \mathbb{R}^m$  and $S \in \mathbb{R}^n \times \mathbb{R}^m$  are semi-algebraic sets, then we see that
$U:=\{ x \in A \ :  (x, y) \in S,\, \forall \ y \in B\}$ is also a semi-algebraic set. To see this, from property {\rm (v)}, we see that  $\{ x \in A \ : \ \exists y \in B,  (x, y) \in S \}$  is semialgebraic. As the complement of $U$
is the union of the complement of $A$  and the set $\{ x \in A \ : \ \exists y \in B,  (x, y) \not\in S \}$, it follows that the complement of $U$ is semi-algebraic by property {\rm (i)}. Thus, $U$ is also semi-algebraic by property
{\rm (i)}. In general, if we have a finite collection of semi-algebraic sets, then any set obtained from them by a finite chain of quantifiers is also semi-algebraic.
}\end{remark}

For a one-dimensional semi-algebraic function, we have further the following properties:
\begin{lemma}{\bf (Monotonicity Theorem \cite{Dries1996})} \label{lemma:5.3}
Let $f$ be a semi-algebraic function $f$ on $\mathbb{R}$.  Let $a,b \in \mathbb{R}$ with $a<b$. Then, there exists a finite
subdivision $a = t_0 < t_1 < \ldots < t_k = b$ such that, on each interval
$(t_i, t_{i+1})$, $f$ is continuous and $f$  either takes a constant value or is strictly monotone.
\end{lemma}

\begin{lemma}{\bf (Growth Dichotomy Lemma \cite{Miller1994})}\label{lemma:5.4}
Let $\epsilon_0>0$ and let $f$ be a continuous semi-algebraic function $f$ on $[0,\epsilon_0]$ with $f(0)=0$. Then either $f$ takes a constant value $0$ over $[0,\epsilon_0]$ or there exist
constants $c \neq 0$ and $p,q \in \mathbb{N}_{>0}$ such that $f(t)=c \, t^{\frac{p}{q}}+o(t^\frac{p}{q})$ as $t \rightarrow 0_+$.
\end{lemma}

\vspace{-0.3cm}

\section*{Appendix B: Proof of Lemma \ref{lemma:2}}

\begin{proof}
[Proof of {\rm (i)}] From the definition of $(P_{\epsilon})$, it is clear that if $0 \le \epsilon_1 \le \epsilon_2$, then ${\rm val}(P_{\epsilon_1}) \ge {\rm val}(P_{\epsilon_2})$. Using a similar method of proof as in Lemma 4.1, one can show that $\epsilon \mapsto {\rm val}(P_{\epsilon})$  is a lower semicontinuous function. Now, let $\epsilon_k \rightarrow \epsilon_{+}$. Then, from the lower semicontinuity,
\[
\liminf_{k \rightarrow \infty}{\rm val}(P_{\epsilon_k}) \ge {\rm val} (P_{\epsilon}).
\]
This together with the fact that $\epsilon \mapsto {\rm val}(P_{\epsilon})$
is nonincreasing shows that
$\lim_{k \rightarrow \infty}{\rm val}(P_{\epsilon_k}) = (P_{\epsilon})$. So, this function is right continuous.

%Semialgebraicity of the function $\epsilon \mapsto {\rm val}(P_{\epsilon})$ follows directly from Tarski-Seidenberg theorem (see, for example, \cite{Benedetti1991, Bochnak1998}).

Let $J(x) := \min_{w}\{G(x, w) \ : \ h_j(w) \le 0, j = 1, \ldots, r\}.$ By
property {\rm (vi)},  $J$ is a semialgebraic function. Let
\begin{eqnarray*}
X & := & \{(\epsilon, x, y) \in [0, +\infty) \times \mathbb{R}^n \times \mathbb{R}^m
\ : \  g_i(x, y) \le 0, i = 1, \ldots, s, \\
&  & \hspace{5cm} h_j(y) \le 0, j = 1, \ldots, r,  G(x, y) - J(x) \le \epsilon \} \\
\mbox{and} & & \\
Y & := & \{(\epsilon, x, y) \in X  : f(x, y) \le f(a, b), \  \forall (\epsilon, a, b) \in X\}.
\end{eqnarray*}
We can verify that $X$ and $Y$ are semialgebraic sets by properties {\rm (ii)}, {\rm (iii)} and Remark \ref{remark:1}. Further, by definition, the graph of the function $\epsilon \mapsto {\rm val}(P_{\epsilon})$ is given by
$\{(\epsilon, f(x, y)) \ : \ (\epsilon, x, y) \in Y\}$.
Clearly, this set is the image of the set $Y$ under the semialgebraic map $(\epsilon, x, y) \mapsto (\epsilon, f(x, y)),$ and hence it is a semialgebraic set by property {\rm (v)}. Thus, $\epsilon \mapsto {\rm val}(P_{\epsilon})$ is a semi-algebraic function on $[0,+\infty)$.

Fix a finite interval $I \subseteq [0,+\infty)$. As $\epsilon \mapsto {\rm val}(P_{\epsilon})$ is a semialgebraic function,
it follows from Lemma \ref{lemma:5.3} that the function $\epsilon \mapsto \mathrm{val}(P_{\epsilon})$
is continuous over $I$ except at finitely many points.
%and so, by Monotonicity Lemma \cite{Bochnak1998}, the last statement follows.

[Proof of {\rm (ii)}] Fix a finite interval $I \subseteq [0,+\infty)$.
%it is continuous over  $[0, +\infty)$ except at finitely many points.
Denote the discontinuity points of $\epsilon \mapsto {\rm val}(P_{\epsilon})$ on $I$ by $\{\epsilon_1,\ldots,\epsilon_l\}$ for some $l \in \mathbb{N}$. Clearly, $\inf_{1 \le i \le l}\epsilon_i>0$ as $\epsilon \mapsto {\rm val}(P_{\epsilon})$ is right continuous. Let $\bar \epsilon=\min_{1 \le i \le l}\{\epsilon_i\}/2>0$. Then, $\epsilon \mapsto {\rm val}(P_{\epsilon})$ is continuous over $[0,\bar \epsilon]$.  Applying Lemma \ref{lemma:5.4} with $f$ replaced by $\epsilon \mapsto {\rm val}(P_{\epsilon})-{\rm val}(P)$ on $[0,\bar \epsilon]$, we see that there exist constants $c>0$, $p,q \in \mathbb{N}_{>0}$ and $\epsilon_0 \in (0,1)$ with $\epsilon_0 < \bar \epsilon$ such that
%the function $\epsilon \mapsto {\rm val}(P_{\epsilon})$ admits a Puiseux expansion at $\epsilon = 0,$ i.e.,
\begin{equation}\label{eq:0.16}
{\rm val}(P_{\epsilon}) \le {\rm val}(P) + c \, \epsilon^{\frac{p}{q}} \le {\rm val}(P) + c \, \epsilon^{\frac{1}{q}}, \quad
\textrm{ for all } \quad \epsilon \in [0,\epsilon_0],
\end{equation}
where the last inequality holds as $0<\epsilon \le \epsilon_0<1$.
%for some $c_k \in \mathbb{R}$ and $q \in \mathbb{N}_{ > 0}.$
 This, together with the nonincreasing property of $\epsilon \mapsto {\rm val}(P_{\epsilon}),$ yields the last assertion.
\end{proof}
\vspace{-0.3cm}
\section*{Appendix C: Proof of Theorem \ref{th:convergence} (Convergence of Algorithm 4.5)}
\begin{proof}
[Proof of {\rm (i)}]  Recall from Lemma \ref{lemma:4.4} that $J_k(x) \le J(x)$ for all $k \in \mathbb{N}$ and for all $x \in \Omega$. So,  ${\rm val}(P_{\epsilon}^k) \ge {\rm val}(P_{\epsilon})$ for all $k \in \mathbb{N}$.
This implies that $v_{\epsilon}^k \ge {\rm val}({P_{\epsilon}})$ for all $k \in \mathbb{N}$. As $v_{\epsilon}^k$ is a non-increasing sequence which is bounded below,
$\lim_{k \rightarrow \infty}v_{\epsilon}^k$ exists. Let $v_{\epsilon}=\lim_{k \rightarrow \infty}v_{\epsilon}^k$.
Then,
\begin{equation}\label{eq:lala0}
v_{\epsilon}  \ge {\rm val}(P_{\epsilon}).
\end{equation}
Let $\delta \in (0,\epsilon)$ and consider problem $(P_{\epsilon-\delta})$. By Assumption 2.1, $K$ and $F$ are compact sets. From the nonsmooth Danskin Theorem (see \cite[Page 86]{Clarke}), we see that the optimal value function of the lower level problem $J(x):=\min_{w \in \mathbb{R}^m}\{G( x,w): h_j(w) \le 0,j=1,\ldots,r\}$ is locally Lipschitz (and so, is continuous). Thus, a global minimizer of $(P_{\epsilon-\delta})$ exists. Let $(\bar x, \bar y)$ be a global minimizer of $(P_{\epsilon-\delta})$.
The set
$D_0:=\{(x,y) \in  K \cap (\mathbb{R}^n \times F): G(x,y)-J(x)< \epsilon, \;  f(x,y)<f(\bar x,\bar y)+\delta\}$
is a nonempty set as $(\bar x, \bar y) \in D_0$. Moreover, from our assumption  we have
$\mathrm{cl}\big(\mathrm{int}(K \cap (\mathbb{R}^n \times F))\big)=K \cap (\mathbb{R}^n \times F)$.
This together with the fact that
$\{(x,y): G(x,y)-J(x)< \epsilon, \;  f(x,y)<f(\bar x,\bar y)+\delta\}$ is an open set gives us that
\[
\tilde{D}:=\{(x,y) \in \mathrm{int}(K \cap (\mathbb{R}^n \times F)): G(x,y)-J(x)< \epsilon  \mbox{ and }  f(x,y)<f(\bar x,\bar y)+\delta\}
\]
is a nonempty open set. So, $D:={\rm Pr}_1 \tilde{D}=\{x \in \mathbb{R}^n: (x,y) \in \tilde{D} \mbox{ for some } y \in \mathbb{R}^m\}$ is also a nonempty open set.
%
% $\{x: G(x,\bar y)-J(x)< \epsilon \mbox{ and }  f(x,\bar y)<f(\bar x,\bar y)+\delta\}$ is an open set, we have ${\rm int} ({\rm Pr}_1 K) \neq \emptyset$ and
% $$D:=\{x \in {\rm int} ({\rm Pr}_1 K): G(x,\bar y)-J(x)< \epsilon \mbox{ and }  f(x,\bar y)<f(\bar x,\bar y)+\delta\}$$
% is an nonempty open set.
%
% We denote  ${\rm Pr}_1 K=\{x \in \mathbb{R}^n: (x,y) \in K \mbox{ for some } y \in \mathbb{R}^m\}$.  Recall the optimal value function of the lower level problem
% $J(x):=\min_{w \in \mathbb{R}^m}\{G(x,w): h_j(w) \le 0,j=1,\ldots,r\}$
% is a continuous function. So, the set
% $D_0:=\{x \in {\rm Pr}_1 K: G(x,\bar y)-J(x)< \epsilon \mbox{ and }  f(x,\bar y)<f(\bar x,\bar y)+\delta\}$
% is an nonempty set (indeed, $\bar x \in D_0$).
Since $J_k$ converges to $J$ in $L^1(\Omega,\varphi)$-norm, $J_k$ converges to $J$ on $\Omega$ almost everywhere. Moreover, as $\varphi (\Omega)<+\infty$,  the classical Egorov's theorem
guarantees that there exists a subsequence
$l_k$ such that $J_{l_k}$ converges to $J$ $\varphi$-almost uniformly on $\Omega$.  So, there exists
 a Borel set $A$ with $\varphi(A)<\frac{\eta}{2}$ with $\eta:=\varphi(D) >0$ such that
$J_{l_k} \rightarrow J \mbox{ uniformly over }  \Omega \backslash A$.
As in the proof of Lemma \ref{lemma:6}, we can show that
$(\Omega \backslash A)\cap D \neq \emptyset$.
Let $x_0 \in (\Omega \backslash A)\cap D$. Then, we have $J_{l_k}(x_0) \rightarrow J(x_0)$ and there exists $y_0 \in F$ such that $G(x_0,y_0)-J(x_0)< \epsilon$ and
$
f(x_0,y_0)<f(\bar x,\bar y)+\delta.
$
 So, for all large $k$,
$G(x_0,y_0)-J_{l_k}(x_0) < \epsilon.$
Thus, for all large $k$, $(x_0,y_0)$ is feasible for $(P^{l_k}_{\epsilon})$ and
\[
v^{l_k}_{\epsilon} \le {\rm val}(P_{\epsilon}^{l_k}) \le f(x_0,y_0) <f(\bar x,\bar y)+\delta={\rm val}(P_{\epsilon-\delta})+\delta.
\]
Letting $k \rightarrow \infty$, we obtain that
$v_{\epsilon}=\lim_{k \rightarrow \infty}v_{\epsilon}^{l_k} \le  {\rm val}(P_{\epsilon-\delta})+\delta$.
Letting $\delta \rightarrow 0^+$, we see that
\begin{equation}\label{eq:lala1}
v_{\epsilon} \le \lim_{\delta \rightarrow \epsilon^{-}}{\rm val}(P_{\delta}).
\end{equation}
Therefore, the inequality ${\rm val}(P_{\epsilon}) \le v_{\epsilon} \le \displaystyle \lim_{\delta \rightarrow \epsilon^{-}}{\rm val}(P_{\delta})$  follows by combining (\ref{eq:lala0}) and (\ref{eq:lala1}).
To see the second assertion in {\rm (i)}, we only need to notice from Lemma \ref{lemma:2}{\rm (i)} that $\epsilon \mapsto {\rm val}(P_{\epsilon})$ is continuous except finitely many points over a finite
interval $I$.

[Proof of {\rm (ii)}] From Lemma \ref{lemma:2}{\rm (ii)}, we see that there exists $\epsilon_0>0$ such that $\epsilon \mapsto {\rm val}(P_{\epsilon})$ is continuous over
$(0,\epsilon_0)$. Thus, from {\rm (i)}, we have
$v_{\epsilon}^k \rightarrow {\rm val}(P_{\epsilon})$ for all $\epsilon \in (0,\epsilon_0)$. Now, fix any $\epsilon \in (0,\epsilon_0)$,
Let $\delta_k \downarrow 0$ as $k\rightarrow \infty$. Let $v_{\epsilon}^{k}=\min_{1 \le i \le k} {\rm val}(P_{\epsilon}^i)={\rm val}(P_{\epsilon}^{i_k})$ and let $(x_k,y_k)$ be a $\delta_k$-solution of $(P^{i_k}_{\epsilon})$.
Then, $\{(x_k,y_k)\} \subseteq K \cap (\mathbb{R}^n \times F)$. As $K$ and $F$ are compact, we see that
$\{(x_k,y_k)\}$ is a bounded sequence. Let $(\hat{x},\hat{y})$ be a cluster point of $\{(x_k,y_k)\}$. Clearly, $(\hat{x},\hat{y}) \in K \cap (\mathbb{R}^n \times F)$.
As $J_k \le J$ on $\Omega$ for all $k \in \mathbb{N}$, $x_k \in {\rm Pr}_1 K \subseteq \Omega$ and $(x_k,y_k)$ is feasible for $(P_{\epsilon}^{i_k})$. Hence, for each $k \in \mathbb{N}$
\[
G(x_k,y_k) -J(x_k) \le G(x_k,y_k) -J_{i_k}(x_k) \le \epsilon.
\]
Passing to the limit and noting that $J$ is continuous,  we get that
$G(\hat{x},\hat{y})-J(\hat{x}) \le  \epsilon$.
So, $(\hat{x},\hat{y})$ is feasible for $(P_{\epsilon})$. Finally, since $v_{\epsilon}^k \rightarrow {\rm val}(P_{\epsilon})$, it follows that
\[
f(\hat{x},\hat{y})=\lim_{k \rightarrow \infty}f(x_k,y_k) \le  \lim_{k \rightarrow \infty}(v_{\epsilon}^k +\delta_k)= {\rm val}(P_{\epsilon})
\]
and $(\hat{x},\hat{y})$ is a global minimizer of $(P_{\epsilon})$.
\end{proof}

{\bf Acknowledgement:} The authors are grateful to the anonymous referees for their insightful
comments and valuable suggestions which have contributed to the final preparation of the paper.

%Let $\hat{\Phi^k}(x,t)=t-\overline{\Phi^k_{\alpha}}(x)$, $\hat{g}_i(x,t)=g_i(x,y)$, $J(x,t)=t$ with $(x,t) \in \mathbb{R}^n \times \mathbb{R}$.
% Consider another  problem (the second stage problem)
% \begin{eqnarray*}
% (P_{2,k}) & \inf_{x \in \mathbb{R}^n} & \overline{\Phi^k}(x) \\
% & \mbox{ subject to} & g_i(x,y) \ge 0, i=1,\ldots,s.
% \end{eqnarray*}
% %\begin{eqnarray*}
% %(P_{2,k}^d) &\displaystyle \inf_{{\bf u} \in \mathbb{N}^d_{n}}& \sum_{\alpha} \overline{\Phi^k_{\alpha}} {{\bf u}}_{\alpha} \\
% %& \mbox{subject to } & {\bf M}_{d}({\bf {u}}) \succeq 0, \ {\bf M}_{d-u_i}({g}_i,{\bf {u}}) \succeq 0, \ i=1,\ldots,s.
% %%& & {\bf M}_{k-u_i}(\hat{\Phi^k},{\bf u}) \succeq 0, {\bf M}_{k-u_i}(\hat{\Phi^{k-1}},{\bf u}) \succeq 0.
% %%& & f-\Phi^k=\sigma_0+ \sum_{i=1}^s \sigma_i g_i, \\
% %%& &¡¡\sigma_i \in \Sigma[x,y]
% %%\langle \sum_{\alpha}p_{\alpha}, {\bf x}^{\alpha}\rangle
% %\end{eqnarray*}
% %where $\hat{\Phi^{0}} \equiv \hat{\Phi^{1}}$.
%
% Step 4: Let $\bar x_k$ be a $(1/k)$-solution of $(P_{2,k})$. Find a solution $\bar y_k$ of the following problem:
% \begin{eqnarray*}
% (P_{3,k}) &\displaystyle \inf_{y \in \mathbb{R}^m} & G(\bar x_k,y) \\
% & \mbox{subject to } &  h_j(\bar x_k,y) \ge 0, j=1,\ldots,r.
% %\langle \sum_{\alpha}p_{\alpha}, {\bf x}^{\alpha}\rangle
% \end{eqnarray*}
%
% Step 5: Let $f_k^*=\min\{f(\bar x_i,\bar y_i): 1 \le i \le k\}$ and let $$(x_k,y_k) \in {\rm argmin}\{f(\bar x_i,\bar y_i): 1 \le i \le k\}.$$
%
% Step 6: Update $k=k+1$ and go back to Step 1.

\end{document}